\documentclass[11pt]{article}

\usepackage[margin=1in]{geometry}
\usepackage{amsmath,amsthm,amssymb,amsfonts,float}

\usepackage[colorlinks=true, pdfstartview=FitV, linkcolor=blue,citecolor=blue, urlcolor=blue]{hyperref}

\usepackage[abbrev,lite,nobysame]{amsrefs}
\usepackage{times}
\usepackage[usenames,dvipsnames]{color}
\usepackage{esint}

\usepackage{mathtools,enumitem,mathrsfs}

\usepackage[compact]{titlesec}

\usepackage[title]{appendix}

\usepackage{comment}

\mathtoolsset{showonlyrefs=true}
    
\newcommand{\ep}{\epsilon}

\newcommand{\leqc}{\lesssim}

\newcommand{\geqc}{\gtrsim}

\newcommand{\grad}{\nabla}

\newcommand{\norm}[1]{\left|\left| #1 \right|\right|}
\newcommand{\abs}[1]{\left| #1 \right|}
\newcommand{\set}[1]{\left\{ #1 \right\}}
\newcommand{\brak}[1]{\left\langle #1 \right\rangle} 

\newcommand{\R}{\mathbb{R}}
\newcommand{\N}{\mathbb{N}}

\newcommand{\C}{\mathbb{C}}

\newcommand{\Z}{\mathbb{Z}}
\newcommand{\T}{\mathbb{T}}
\newcommand{\K}{\mathbb{K}}
\renewcommand{\S}{\mathbb{S}}

\newcommand{\loc}{\mathrm{loc}}

\newcommand{\Hbf}{{\bf H}}
\newcommand{\Wbf}{{\bf W}}

\newcommand{\dee}{\mathrm{d}}
\newcommand{\ds}{\dee s}
\newcommand{\dt}{\dee t}

\newcommand{\dx}{\dee x}

\newcommand{\dy}{\dee y}

\newcommand{\dr}{\dee r}

\DeclareMathOperator{\Div}{\mathrm{div}}
\DeclareMathOperator{\Id}{\mathrm{Id}}

\usepackage{bbm}

\renewcommand{\P}{\mathbf{P}}

\newcommand{\E}{\mathbf{E}}
\newcommand{\EE}{\mathbf E}
\newcommand{\PP}{\mathbf P}

\newtheorem{theorem}{Theorem}[section]
\newtheorem{proposition}[theorem]{Proposition}
\newtheorem{corollary}[theorem]{Corollary}
\newtheorem{lemma}[theorem]{Lemma}
\newtheorem*{lemma*}{Lemma}

\newtheorem{assumption}{Assumption}
\newtheorem{system}{System}

\theoremstyle{definition}
\newtheorem{definition}[theorem]{Definition}
\newtheorem{remark}[theorem]{Remark}

\numberwithin{equation}{section}
    
\begin{document}

\title{The Batchelor spectrum of passive scalar turbulence\\ in stochastic fluid mechanics at fixed Reynolds number}
\author{Jacob Bedrossian\thanks{\footnotesize Department of Mathematics, University of Maryland, College Park, MD 20742, USA \href{mailto:jacob@math.umd.edu}{\texttt{jacob@math.umd.edu}}. J.B. was supported by NSF CAREER grant DMS-1552826 and NSF RNMS \#1107444 (Ki-Net)} \and Alex Blumenthal\thanks{\footnotesize Department of Mathematics, University of Maryland, College Park, MD 20742, USA \href{mailto:alex123@math.umd.edu}{\texttt{alex123@math.umd.edu}}. This material was based upon work supported by the National Science Foundation under Award No. DMS-1604805.} \and Sam Punshon-Smith\thanks{\footnotesize Division of Applied Mathematics,  Brown University, Providence, RI 02906, USA \href{mailto:punshs@brown.edu}{\texttt{punshs@brown.edu}}. This material was based upon work supported by the National Science Foundation under Award No. DMS-1803481.}}

\maketitle

\begin{abstract}

In 1959, Batchelor predicted that the stationary statistics of passive scalars advected in fluids with small diffusivity $\kappa$ should display a $|k|^{-1}$ power spectrum along an inertial range contained in the viscous-convective range of the fluid model. This prediction has been extensively 
tested, both experimentally and numerically, and is a core prediction of \emph{passive scalar turbulence}.

In this article we provide a rigorous proof of a version of Batchelor's prediction in the $\kappa \to 0$ limit when the scalar is subjected to a spatially-smooth, white-in-time stochastic source and is advected by the 2D Navier-Stokes equations or 3D hyperviscous Navier-Stokes equations in $\T^d$ forced by sufficiently regular, nondegenerate stochastic forcing. Although our results hold for fluids at arbitrary Reynolds number, this value is fixed throughout. 
%The scalar is subjected to a smooth-in-space, white-in-time stochastic source, and evolves by advection-diffusion with diffusivity $\kappa > 0$. 
Our results rely on the quantitative understanding of Lagrangian chaos and passive scalar mixing established in our recent works. 
%our recent works \cite{BBPS18, BBPS19I, BBPS19II}. 
Additionally, in the $\kappa \to 0$ limit, we obtain statistically stationary, weak solutions in $H^{-\ep}$ to the stochastically-forced advection problem \emph{without diffusivity}. These solutions are almost-surely not locally integrable distributions with non-vanishing average anomalous flux and satisfy the Batchelor spectrum at all sufficiently small scales. We also prove an Onsager-type criticality result which shows that no such dissipative, weak solutions with a little more regularity can exist.
\end{abstract}

\setcounter{tocdepth}{2}
{\small\tableofcontents}

%% Introduction

\section{Introduction}\label{sec:Intro}
%!TEX root = master.tex
%\section{Introduction to passive scalar turbulence} \label{sec:PST}
% Let $u_t(x)$ be a velocity field evolving according to a fluid evolution equation in the $d$-dimensional periodic box $\T^d = [0,2\pi]^d$, e.g. Navier-Stokes. 

A fundamental problem in fluid mechanics concerns the behavior of the concentration of a scalar $g_t^\kappa$ being passively advected by an incompressible fluid velocity $u_t$ while also undergoing some small amount of molecular diffusion. In many circumstances the scalar exhibits complex, (approximately) statistically self-similar patterns over a range of scales, referred to as \emph{passive scalar turbulence} (see e.g. \cite{NyeBrodkey67,FalkovickEtAl01,ShraimanSiggia00,W00,MK99} for more discussions of the physical background and Remark \ref{rem:turb} for some context for the use of the word `turbulence'). 

% \emph{Passive scalar turbulence} refers to the dynamics of a passive tracer $g_t$ being advected by an incompressible fluid motion $(u_t)$ as the diffusivity $\kappa>0$ is taken to zero.  

In this work, we consider a fluid in the periodic box $\T^d, d = 2$ or $3$, where the fluid velocity $u_t$ evolves according to a time-continuous, randomly driven ergodic motion. The scalar $g_t^\kappa$ solves the following advection diffusion equation with stochastic source: 
\begin{equation}
% \label{eq:ad}
\begin{aligned}
% &\partial_t u_t + u_t \cdot \grad u_t + \grad p_t - \nu \Delta u_t = Q\dot{W}_t \\ 
&\partial_t g_t^\kappa + u_t \cdot \grad g_t^\kappa - \kappa \Delta g_t^\kappa  = \dot{s}_t \, .
\end{aligned}
\end{equation}
Here $\kappa > 0$ is the molecular diffusivity, and the source $\dot{s}_t$ is a white-in-time, Gaussian process supported at low spatial frequencies. For simplicity, it suffices to consider ${s}_t(x) = b(x) {\beta}_t$, where $b$ is a non-zero smooth function on $\T^d$ and $\beta_t$ is a 1D Brownian motion independent from $u_t$. We will assume that the velocity $u_t$ itself evolves according to an incompressible, stochastically forced fluid model, for instance the {\em stochastic Navier-Stokes equations} on $\T^2$,
\[
\begin{aligned}
	&\partial_t u_t + u_t\cdot\nabla u_t + \nabla p_t - \nu \Delta u_t = Q\dot{W}_t \, , \\
    &\Div u_t = 0 \, , 
\end{aligned}
\]
or the {\em stochastic hyper-viscous Navier-Stokes equations} on $\T^3$, both of which are known to be ergodic with a unique stationary measure under fairly mild non-degeneracy assumptions on the stochastic forcing (see Sections \ref{sec:FluidMod} for more precise details). We also consider a variety of other finite dimensional fluid models with better time and space regularity (see Section \ref{sec:CktCx}).

% Here, $QW_t$ is a white-in-time, sufficiently-regular-in-space, Gaussian process (see Section \ref{sec:Prelim} for details), $b \in C^\infty(\T^d)$, and $\xi_t$ is an independent 1D Brownian motion. Our methods apply equally well when the random source is given by $\sum_{j=1}^\infty b_j \xi_t^{j}$ for a sequence of smooth, $L^2$ orthogonal $b_j$'s that decay sufficiently fast in $j$ and independent Brownian motions $\xi_t^{j}$. 
% For simplicity we will consider the case $x\in \mathbb T^d$; see Section \ref{sec:Future} for discussion.
The presence of the source $s_t$ and the ergodicity of $u_t$ allows for the scalar to settle into a {\em statistical steady state} with ensemble $\E$, where the scalar input from the source $\dot{s}_t$ is balanced by dissipation due to the diffusion from $\kappa \Delta$ and the law of $g_t^\kappa$ is the same for all times. Specifically, when ${s}_t = b{\beta}_t$, It\^{o}'s formula implies that in a statistical steady state, there is an average constant dissipation rate $\chi := \|b\|_{L^2}^2 >0$,
\[\label{eq:L2BalanceIntro}
2\kappa \E\| \grad g^\kappa\|_{L^2}^2 =  \chi
\]
providing a mechanism for a {\em cascade} to high frequencies (see Proposition \ref{prop:StatMes} for details). 

When $\kappa$ is taken very small, complex (approximately) self-similar patterns in the scalar $g_t^\kappa$ emerge. In a statistical steady state these patterns usually have an $L^2$ power spectrum that approximates a {\em power law} over a certain range of frequencies. For physical fluid flows and in numerical simulations of the Navier-Stokes equations, such power laws are frequently observed, with different regimes depending on the size of the Schmidt number $\operatorname{Sc} := \operatorname{Pr}/\operatorname{Re}$, where $\operatorname{Re} = \nu^{-1}$ is the Reynold's number, and $\operatorname{Pr} = \kappa^{-1}$ is the Prandtl number.
In his foundational paper \cite{Batchelor59}, Batchelor predicted that when $\operatorname{Sc} \gg 1$ (known as the Batchelor regime), the steady-state statistics for a passive scalar exhibit a $|k|^{-1}$ power spectrum over the {\em viscous-convective subrange} of frequencies, known as \emph{Batchelor's law} (or the \emph{Batchelor spectrum}):\footnote{Batchelor's original prediction was actually for scalars on $\R^3$ (see Section \ref{sec:Batchelors-arg-review} for more details) instead of the periodic box, the prediction below is adapted to the periodic setting} 
\begin{equation}\label{eq:Batch1}
  \Gamma(|k|) : = \abs{k}^{d-1}\EE \abs{\hat{g}^\kappa(k)}^2  \approx \chi\abs{k}^{-1} \quad \textup{ for }\quad (\ell^{NSE}_D)^{-1} \ll \abs{k} \lesssim \kappa^{-1/2} \, , 
\end{equation}
where for each $k \in \Z^d$, $\hat{g}^\kappa(k)$, denotes the Fourier transform of $g^\kappa$ on $\T^d$ and $\ell_D^{NSE} = (\nu/\ep)^{d/4}$ is the dissipative range for the Navier-Stokes equations with $\ep$ being the fluid energy dissipation rate.

\begin{figure}[H]
\begin{center}
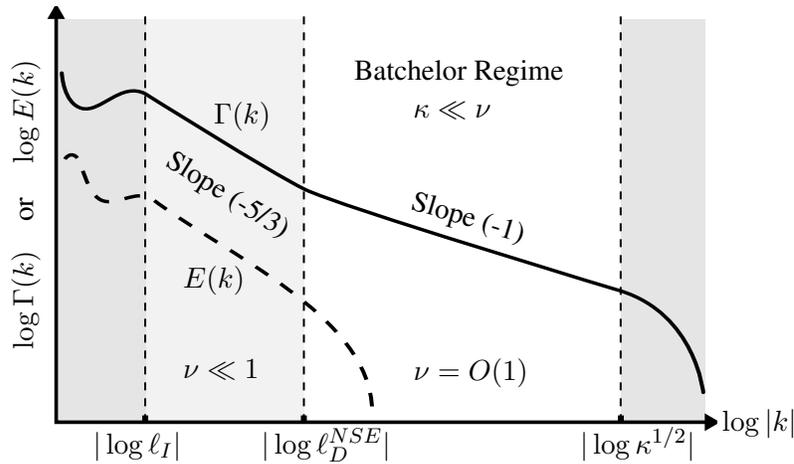
\end{center}
\caption{
The solid line above is the full prediction for the power spectrum $\Gamma(k) = \abs{k}^{d-1}\EE \abs{\hat{g}^\kappa(k)}^2$ of a statistically stationary passive scalar $g^\kappa$ advected by a \emph{turbulent} fluid $u_t$ in 3D. 
In the Batchelor regime $\kappa \ll \nu$, it is theorized that the power law $|k|^{-1}$
holds for scales between $\ell_D^{NSE}$ (the dissipative scale for Navier-Stokes) and $\kappa^{1/2}$ (the dissipative scale for Batchelor's regime).  Between $\ell_I$ and $\ell_{D}^{NSE}$ one expects \emph{Obukhov-Corrsin statistics}, e.g., the $|k|^{-5/3}$ law in dimension $d = 3$, as depicted by the dashed curve above ($E(k) = \abs{k}^{d-1} \EE \abs{\hat{u}(k)}^2$).
 No universality is expected at length-scales in the integral range above $\ell_I$. We emphasize that in the setting of the present manuscript, $\nu$ is fixed $O(1)$ and $\kappa \ll 1$; as such, even if $\nu$ itself is quite small, our methods do not distinguish the Obukhov-Corrsin regime from the integral range.
}
\label{fig:PowerLaw}  
\end{figure}

Since Batchelor made his prediction, engineers and physicists have been making measurements of the spectrum in nature, e.g., temperature and salinity variations in the ocean (see e.g. \cite{NyeBrodkey67,GrantEtAl68,DC80} and the references therein), in laboratory experiments (see e.g. \cite{GibsonSchwarz63,WuEtAl95,WMG97,MD96,JCT00,AK04} and the references therein), and in numerical studies (see e.g. \cite{HZ94,BDY97,YuanEtAl00,DSY10,AO03} and the references therein). As discussed in e.g. \cite{DSY10,AO03,MK99} it remains an open problem in physics to determine the settings in which the Batchelor spectrum is expected and to what degree of accuracy it holds. We will briefly discuss some theoretical studies in the physics literature in Section \ref{subsec:mechanismBatch}.

In this paper, we are primarily concerned with the \emph{Batchelor regime} at \emph{fixed Reynolds number}: where $\nu$ is considered fixed at an arbitrary (potentially small) number and $\kappa \ll \nu$. It is important to note that at length scales below the dissipative scale $\ell_{D}^{NSE}$ for the fluid (or equivalently, frequencies above $(\ell_D^{NSE})^{-1}$) the velocity field $u_t$ is expected to be smooth while only $g_t$ becomes rougher and rougher as $\kappa \to 0$, which makes it significantly easier to understand the ``nonlinear''\footnote{Nonlinear here refers to nonlinearity of the mapping $(u,g) \mapsto u\cdot \nabla g$} advection term $u_t \cdot \grad g_t$ in low regularity.

% While not ``easy'' (the bulk of the necessary analysis is carried out in our previous work \cite{BBPS18,BBPS19I,BBPS19II}), it is nevertheless much more approachable at the current time than understanding turbulence in 2D or 3D incompressible Navier-Stokes as $\nu \to 0$.

The Batchelor spectrum is, in some sense, predicated on smoothness of the velocity field -- it only holds over length scales relative to which the velocity field is essentially smooth.
% The Batchelor spectrum arises from lower frequencies in the velocity interacting with much higher frequencies in the scalar.  
This can be seen from Batchelor's original argument \cite{Batchelor59}; see further discussions in \cite{BMOV05,AntonsenEtAl96,AntonsenOtt1991,FalkovickEtAl01,BalkovskyFouxon99} and below in Section \ref{subsec:mechanismBatch}.
In this regard, the much simpler case of fixed Reynolds number seems to be a reasonable first place to begin any mathematical study of Batchelor-regime passive scalar turbulence, and indeed, of passive scalar turbulence in general.  
% It is important to point out that Batchelor made this prediction for an inertial range extending from the dissipative ranges of the Navier-Stokes equations, $\epsilon^{3/4} \nu^{-3/4} \lesssim \abs{k} \lesssim \kappa^{-1/2}$ in three dimensions and $\epsilon^{1/2} \nu^{-1/2} \lesssim \abs{k} \lesssim \kappa^{-1/2}$ in two dimensions. 
For frequencies below $(\ell_D^{NSE})^{-1}$, hydrodynamic turbulence gives rise to a different approximate power law (a $|k|^{-5/3}$ law for $d=3$), known as Obukhov-Corrsin spectra for the passive scalar \cite{O1949,Corrsin1951}; see e.g. \cite{ShraimanSiggia00,AO03}.
%However, this is clearly out of reach mathematically at the moment, as a mathematically rigorous description of hydrodynamic turbulence, i.e. turbulence of 2D and 3D Navier-Stokes is completely open right now. 
See Figure \ref{fig:PowerLaw} and the associated caption for a description of these two different expected regimes of passive scalar turbulence.

If one does not fix the Reynolds number and simultaneously takes $\nu, \kappa \to 0$ , while sending $\operatorname{Sc} \to \infty$, then the fluid itself is becoming turbulent while nevertheless, the viscous-convective subrange remains in the Batchelor regime. This is a situation that we {\em cannot} treat in this work. Indeed, since almost nothing rigorous has been established mathematically for the (very singular) turbulent limit $\nu \to 0$ for statistically stationary solutions $u_t$ to the Navier-Stokes equations\footnote{For example, it is unknown whether or not the kinetic energy remains bounded in 3D.}, this appears well out of reach of rigorous mathematical analysis for the time being. One model where this kind of `turbulent advection' has been studied is in the \emph{Kraichnan model}, in which $u_t$ is replaced by an idealized rough-in-space, white-in-time Gaussian velocity field, and the equation for $g_t$ is interpreted as a Stratonovich stochastic transport equation. This model was introduced for this purpose in \cite{Kraichnan68} and there is now a wide literature on this model in physics (see \cite{ShraimanSiggia00, CrisantiEtAl1991, crisanti2012products} and the references therein). Properties of stochastic transport with rough velocities have also been studied mathematically as well in the work on isotropic stochastic flows \cite{Jan2002-je,Baxendale1986-nc,Luo2008-jn} and well-posedness by stochastic perturbation \cite{Flandoli2009-pa,Fedrizzi2013-qe,Mohammed2015-fp}. See also the recent preprint \cite{DEGI19} for a deterministic work on turbulent advection by rough velocities.

\begin{remark}\label{rem:turb}
In the physics and engineering literature, `turbulence' is used to describe a wide range of observed phenomena in dynamics of weakly-damped, infinite-dimensional, nonlinear conservative systems in ``generic'' settings, including both 2D and 3D incompressible or compressible Navier-Stokes, passive scalars advected by fluids, nonlinear dispersive equations, magneto-hydrodynamics models, and kinetic models in plasmas \cite{Frisch1995,BE12,W00,ShraimanSiggia00,Lele94,Biskamp2003,Nazarenko11,ZLF12,GrovseljEtAl17}. 

In the Batchelor regime, the passive scalar model we consider is one of the simplest settings in applications where turbulent phenomena can be observed that shares many similarities with hydrodynamic turbulence \cite{ShraimanSiggia00}, specifically: a frequency cascade, ``anomalous'' dissipation, power-laws on the power spectrum and scaling laws on structure functions. 
\end{remark}

\subsection{Review of Batchelor's argument} \label{sec:Batchelors-arg-review}

In \cite{Batchelor59} Batchelor made his prediction by studying the effect of advection and diffusion on pure Fourier modes when the velocity field is a linear ``pure straining flow'', namely a linear velocity field $u_t(x) = Ax$ on $\R^3$ whose matrix $A$ is traceless with real distinct eigenvalues. He argued that, upon zooming in on the velocity field all the way down to the viscous-convective sub-range, the flow is most likely to be approximated by a pure straining flow (see Section \ref{subsec:mechanismBatch} for more discussion on the validity of this approximation). For such flows he showed that frequencies increase exponentially fast, and in a steady state exhibit an exact asymptotic formula for the power spectral density.

To understand the essence of this approach, let us consider the case $d=2$ and take the velocity field $u_t(x) = Ax$ on $\R^2$, where
\[
	A = \begin{pmatrix}
	\gamma & 0 \\ 0 & -\gamma 
	\end{pmatrix}.
\]
If one takes a pure Fourier mode as initial data $g_0(x) = \sin(\xi_0\cdot x)$ for the advection diffusion equation without a source then the solution takes the form $C^\kappa_t \sin (\xi_t\cdot x)$, where the amplitude $C_t^\kappa$ and frequency $\xi_t$ solve the ODE system 
\[
\dot C_t^\kappa = - \kappa |\xi_t|^2 C_t^\kappa, \quad \dot{\xi}_t = - A\xi_t.
\]
If the initial frequency $\xi_0\in \R^2$ has a non-trivial projection onto $\hat{e}_1 = (1,0)$, say for example $|\langle \xi_0, \hat{e}_1\rangle| = 1$, then there is an exponential-in-time increase in frequency $|\xi_t| \sim e^{\gamma t}$, indicating a {\em cascade} from low to high frequencies. At the same time, the amplitude decays double exponentially fast due to diffusion, to wit, $C_t^\kappa \sim \exp(-\kappa e^{2\gamma t}/2\gamma)$, since the diffusion acts more strongly at high frequencies.

To see what effect this has on the steady-state power-spectrum, consider a stochastic source $\dot{s}_t = \sin(\xi_0\cdot x)\dot{\beta}_t$ supported at frequency $\xi_0$. In a statistical steady state, the scalar field $g^\kappa$ is equal in law to the stochastic integral
\[
	g^\kappa(x) = \int_0^\infty C_s^\kappa  \sin(\xi_s \cdot x) \dee \beta_s \, 
%	g^\kappa = \int_0^\infty\exp\left(- \frac{\kappa}{2\gamma} e^{2\gamma s} \right) \sin(\xi_s \cdot x) \dee \beta_s \, 
\]
which is a Gaussian random variable in $L^2_\loc$. Using It\^{o}'s isometry, and the fact that $\chi : = \lim_{R\to \infty}\fint_{[0,R]^2} |\sin(\xi_s \cdot x)|^2\dx$ does not depend on $s$, we find that the average $L^2$ mass density carried by frequencies less than $n \geq 0$ is given by
\[
\begin{aligned}
	 \langle|\Pi_{\leq n}g^\kappa|^2\rangle &:= \lim_{R\to\infty}\fint_{[0,R]^2}\E|\Pi_{\leq n}g^\kappa(x)|^2\dx\\ 
	&= \chi \int_0^{t(n)} C_s^{2\kappa}\ds,
%& = \chi \int_0^{t(n)} \exp\left(- \frac{\kappa}{2\gamma} e^{2\gamma s} \right)  \dee s \, , 
\end{aligned}
\]
where $\Pi_{\leq n}$ denotes the projection onto frequencies $|\xi|\leq n$ and $t(n)$ is the unique time such that $|\xi_{t(n)}| = n$ (which exists for large enough $n$ by monotonicity of $t\mapsto |\xi_t|$). The power-spectral density is then defined by
\[
	\Gamma(n) := \frac{\dee}{\dee n} \langle|\Pi_{\leq n}g^\kappa|^2\rangle = t^\prime(n) \chi C^{2\kappa}_{t(n)}.
    %\Gamma(n) := \frac{\dee}{\dee n} \langle|\Pi_{\leq n}g^\kappa|^2\rangle = t^\prime(n) \frac{\chi}{2\gamma} \exp\left(- \kappa e^{2\gamma t(n)} \right). 
\] 
That $|\xi_t| \sim e^{\gamma t}$ implies the asymptotic $t(n) \sim \log{n}/\gamma$ and $t^\prime(n) \sim \frac{1}{\gamma n}$, therefore 
\begin{equation}
\begin{aligned}
	\Gamma(n) % &\sim \frac{\chi}{\gamma n}\exp\left(-2\kappa\int_0^{\log n/\gamma} e^{2\gamma s}\ds \right)\\
	 &\sim \frac{\chi}{\gamma n}\exp\left(- \frac{\kappa n^2}{\gamma}\right). 
\end{aligned}
\end{equation}
We recover Batchelor's prediction $\Gamma(n) \sim \frac{\chi}{\gamma n}$ in the range $1 \ll n\leqc (\gamma/\kappa)^{1/2}$.

\subsection{Going beyond pure straining flows}\label{subsec:mechanismBatch}

%\begin{itemize}
%	\item Pure straining flow clearly not a reasonable model of Lagrangian flow 
%	in realistic, spatially regular fluids, for which phase portrait has much more variety, giving
%	a combination of shearing, coherent structures, vortices etc. 
%	\item Remarkable however that Batchelor's prediction is valid and robust as it is. 
%	\item As it turns out, pure strain flow is a reasonable model of \emph{moving frame
%	behavior along typical passive tracer (Lagrangian) trajectories}. This moving-frame 
%	saddle-type behavior, known as \emph{hyperbolicity} in dynamical systems theory,
%	is responsible for the generation of small scales in passive scalars. 
%\end{itemize}

%Pure straining flow is clearly not a reasonable model of fluid motion at small scales, but 
%remarkably Batchelor's prediction is not only valid but surprisingly robust. 

It is natural to question whether a time-stationary pure straining flow is actually a good approximation of a fluid at small scales. Indeed, since the velocity fields changes in time, so too  does the approximating linear flow (i.e., the gradient), so one should at least expect a linear phase portrait changing in time. Moreover, there are many coherent Lagrangian structures (e.g., vortices) which locally exhibit shearing or rotational phase portraits, incompatible with the expansion/contraction exhibited by the pure strain flow. Remarkably, however, Batchelor's argument gets the right answer; in fact, Batchelor's law is a robust prediction, holding for a large class of {\em smooth} incompressible flows not necessarily arising from a physical fluid mechanical model. 

The purpose of this paper is to show that a ``cumulative'' version of Batchelor's prediction (see Theorem \ref{thm:Batch} for an exact statement) is a consequence of the {\em chaotic mixing} properties of $u_t$ proved in our previous works \cite{BBPS18,BBPS19I,BBPS19II}.

\subsubsection{Lagrangian chaos}

In short, Batchelor's argument succeeds because, even though the flow $u_t$ is not always well-approximated by a pure straining flow, the linearized time-$t$ motion of passive tracers (i.e., time-$t$ Lagrangian flow) \emph{does} resemble a pure strain flow for large times $t$. 
To make this more explicit, recall that in the absence of diffusivity ($\kappa = 0$) and sources ($s_t \equiv 0$), the scalar $g_t$ is given by 
$g_t=g_0\circ(\phi^t)^{-1}$, where the {\em Lagrangian flow map} $\phi^t:\T^d \to \T^d$ describes the position of a particle $x_t = \phi^t(x)$ starting from $x\in \T^d$ and solves the ODE
\[
	\frac{\dee}{\dt}\phi^t(x) = u_t(\phi^t(x)) \, \quad \phi^0(x)=x \, .
\]
The derivative of the flow $D\phi^t(x):\R^d \to \R^d$ then solves the linearized equation
\[
	\frac{\dee}{\dt}D\phi^t(x) = Du_t(\phi^t(x))D\phi^t,\quad D\phi^0(x)=\Id.
\]

By incompressibility (note $\det D \phi^t (x) \equiv 1$),  growth in time of $| D\phi^t(x)|$ for typical $x$ is associated with the development of strongly expanding and contracting directions of $D \phi^t(x)$ for each fixed $t$, features resembling those of the linear pure strain flow phase portrait. 
Growth of $|D \phi^t(x)|$ can be quantified in terms of positivity of the Lyapunov exponent
\begin{align}\label{defn:lyap}
\lambda_1(x) = \limsup_{t \to \infty} \frac{1}{t} \log |D\phi^t(x)| > 0
\end{align}
for `typical' $x \in \T^d$, which we refer to as {\em Lagrangian chaos}. Previously Lagrangian chaos has been proved for stationary, white-in-time velocity fields 
in \cite{baxendale1993kinematic, baxendale1989lyapunov}. 
In \cite{BBPS18}, we proved that the Lagrangian flow associated to solutions $u_t$ of the stochastically-forced 
2D Navier-Stokes and 3D hyperviscous Navier-Stokes in the sense that there is a deterministic constant $\lambda_1 > 0$ such that \eqref{defn:lyap} holds with lim-sup replaced with lim for all $x$ and all initial fluid configurations, almost surely (see \cite{BBPS18} for rigorous statements). 

As has been observed by a long line of previous authors (see, e.g., \cite{BMOV05,AntonsenEtAl96,AntonsenOtt1991,FalkovickEtAl01,BalkovskyFouxon99}), 
Lagrangian chaos is associated with a low-to-high transfer of $L^2$ mass of passive scalars. Indeed, without diffusivity ($\kappa = 0$)
and sources ($\dot s_t = 0$) we clearly have that $\nabla g_t \circ \phi^t = (D \phi^t)^{-\top} \nabla g_0$, hence one expects $\| \nabla g_t\|_{L^2}$ to grow exponentially fast in time. In view of $L^2$ conservation ($\| g_t\|_{L^2} = \| g_0\|_{L^2}$), this is strongly reminiscent of the transfer of $L^2$ mass from low to high frequencies appearing in Batchelor's original argument for pure shear-strain flow. 

We emphasize that the results of \cite{BBPS18} and the proof of \eqref{defn:lyap} rely crucially upon the stochastic framework: it is often very difficult to provide rigorous proofs of
positivity of Lyapunov exponents, even for deceptively simple models such as the Chirikov standard map family \cite{duarte1994plenty, gorodetski2012stochastic} (a discrete-time toy model of the Lagrangian flow \cite{CrisantiEtAl1991}). See \cite{pesin2010open, young1995ergodic} for more discussions. 
In Navier-Stokes, one of the enemies is the formation of coherent vortices inside of which hyperbolicity is halted (see e.g. \cite{BabianoProvenzale07}). The arguments in our paper \cite{BBPS18} imply that with probability 1, vortices of this kind cannot permanently trap any particles.

\subsubsection{Uniform in diffusivity chaotic mixing} 

%A positive Lyapunov exponent alone is not enough to obtain the kind of frequency cascade that led to Batchelor's law in the CAT map example above.

For the fluid models considered in this paper, a positive Lyapunov exponent alone is not enough to prove Batchelor's law.
Lagrangian chaos only implies that some ``scalar energy'' is going to high frequencies, i.e. some small scales are being created.
However, we need here that small scales are being created everyhere in a reasonably uniform way with high probability (i.e. with explicit moment bounds on fluctuations in the rate of small scale creation). 
To quantify this, we use a much stronger property: uniform-in-$\kappa$, almost-sure exponential mixing\footnote{One can construct dynamical systems with a positive Lyapunov exponent but arbitrarily slow (e.g., polynomial or logarithmic) mixing/decay of Lagrangian correlations, for example, Pommeau-Manneville maps (see, e.g., \cite{liverani1999probabilistic, sarig2002subexponential}).}.
In \cite{BBPS19II}, we proved that if $\bar g_t^\kappa$ is mean zero and solves the initial value problem
\begin{align}
& \partial_t \bar g_t^\kappa + u_t\cdot \grad \bar g_t^\kappa- \kappa \Delta \bar g_t^\kappa =0 
\end{align}
with diffusivity but \emph{no} random source, then there exists a \emph{deterministic} constant $\gamma > 0$ \emph{independent of $\kappa$} and random constant $D_\kappa$ (depending on initial fluid configuration) such that %the negative Sobolev norm
\begin{align}
	\|\bar g_t^\kappa\|_{H^{-1}} := \sup_{\|f\|_{H^1} =1 }\left|\int f \,\bar g_t^\kappa\, \dx\right| \leq D_\kappa e^{-\gamma t}\|\bar g_0\|_{H^1}, \label{eq:k=0-mixing}
\end{align}
where $D_\kappa$ also has suitable moment bounds \emph{independent of $\kappa$} (see Definition \ref{defn:randConst} and Theorem \ref{thm:UniMix} for precise statements). See \cite{Thiffeault2012-xs} for a discussion of using negative Sobolev norms to quantify mixing. There is a large mathematical literature on scalar mixing in the mathematics literature; see e.g. \cite{Bressan03,LDT11,S13,IyerXu14,YZ17,TZ18,GGM19,DEGI19} and the references therein.

To see why uniform-in-$\kappa$ scalar mixing implies an exponential increase of frequency scale, note that 
\begin{align}
\norm{\Pi_{\leq N} \bar g_t^\kappa}_{L^2} \leq N \norm{\bar g_t^\kappa}_{H^{-1}} \leq D_\kappa N e^{-\gamma t} \norm{\bar g_0}_{H^1}  \, , 
\end{align}
where $\Pi_{\leq N}$ denotes projection to Fourier modes of frequency $\leq N$.
From this, one can see that for times $t \gg \frac{\log N}{\gamma}$, most of the scalar has been transferred from frequencies $\leq N$ to higher frequencies.\footnote{ The appearance of the mixing rate $\gamma$ here in the estimate $t \gg \frac{\log N}{\gamma}$ for the low-to-high transfer is suggestive of why 
Lagrangian chaos alone is insufficient to prove Batchelor's law.}
As in 
%the CAT map example and 
Batchelor's original argument, this exponentially fast transfer from low to high frequencies is exactly the mechanism which gives rise to Batchelor's law when $u_t$ is given by Navier-Stokes -- one difficulty is of course dealing with potential unboundedness of the random constant $D$,
which captures fluctuations in the mixing time.

Our work \cite{BBPS19II} builds on our earlier work \cite{BBPS19I} that proved the corresponding statement for $\kappa = 0$; similarly, this latter work uses the Lagrangian chaos result of \cite{BBPS18} as a lemma.
Our works \cite{BBPS19I,BBPS19II} are based on analyzing two-point statistics of the Lagrangian flow and a key step is to upgrade the positivity of the Lyapunov exponent to positivity of the \emph{moment Lyapunov exponents}
\begin{align}
\Lambda(p) = - \lim_{t \to \infty} \frac{1}{t} \log \EE |D\phi^t(x)|^{-p} \, , \quad 0 < p \ll 1.
\end{align}
This is deeply related to large deviations of the finite-time Lagrangian Lyapunov exponents as $t\to \infty$ (for the relation of moment Lyapunov exponents to large deviations in the convergence of Lyapunov exponents, see, e.g., \cite{arnold1984formula, arnold1987large, arnold1986lyapunov}). 
It had already been realized by physicists that fluctuations of Lagrangian Lyapunov exponents should play a key role in Batchelor's law (see discussions in \cite{AntonsenOtt1991,AntonsenEtAl96,BalkovskyFouxon99} and the references therein) and so our works are in some ways a mathematically rigorous completion of some of these ideas.

%!TEX root = master.tex

\subsection{Main results}\label{subsec:main-results}
% In all our results we will consider our domain to be the $d$-dimensional periodic box $\T^d = [0,2\pi]^d$ for $d= \{2,3\}$. 

We now turn to detailed statements of the main results of this paper. 
After providing some preliminary definitions and conventions, in 
Section \ref{subsubsec:cumBatchSpec} we state a ``cumulative'' version of 
Batchelor's power law spectrum for fluids at fixed, finite Reynolds number. 
Section \ref{subsubsec:Yaglom} describes a version of Yaglom's law, a 
scaling law analogous to the $-4/5$ law in hydrodynamic turbulence. 
In Section \ref{sec:ADOC} and \ref{subsubsec:oCritSubsc} we turn our attention
to the description of \emph{ideal passive scalar turbulence}, i.e., the description 
of a class of low-regularity solutions to the ``inviscid'' $\kappa = 0$ advection-diffusion equation exhibiting a scale-by-scale flux of $L^2$ mass from low to high modes. 

The velocity field $(u_t)$ will take values in the space $\Hbf$ consisting mean zero divergence free velocity fields belongs to $H^\sigma = H^\sigma(\T^d;\R^d)$, the Sobolev space of mappings from $\T^d\to \R^d$ with regularity $\sigma > \frac{d}{2}+3$; note that this implies velocities are always at least $C^3$ in space. The {\em velocity process} $(u_t)$ will evolve in $\Hbf$ according to one of the following stochastic PDEs depending on whether $d=2$ or $3$: 
\begin{system}[2D Navier-Stokes equations]\label{sys:NSE}When $d=2$, $(u_t)$ solves
\begin{equation}
\begin{aligned}
&\partial_t u_t + u_t \cdot \grad u_t =- \grad p_t + \nu \Delta u_t + Q \dot W_t \\ 
&\Div u_t = 0 \, , 
\end{aligned}
\end{equation}
where $u_0 = u \in \Hbf$. Here, the viscosity $\nu >0$ is a fixed constant.
\end{system}

\begin{system}[3D hyper-viscous Navier-Stokes]\label{sys:3DNSE}When $d=3$, $(u_t)$ solves
\begin{equation}
\begin{aligned}
&\partial_t u_t + u_t \cdot \grad u_t =- \grad p_t - \nu \Delta^{2} u_t + Q \dot W_t \\ 
&\Div u_t = 0,
\end{aligned}
\end{equation}
where $u_0 = u \in \Hbf$. Here, the hyperviscosity $\nu>0$ is a fixed constant.
\end{system}
In the above systems, $W_t$ is a cylindrical Wiener process on mean-zero, divergence free $L^2$ vector fields with respect to an associated canonical stochastic basis $(\Omega_W,\mathscr{F}^W,(\mathscr{F}^W_t),\P_W)$ and $Q$ a positive Hilbert-Schmidt operator on mean-zero, divergence free $L^2$ vector fields satisfying suitable non-degeneracy and regularity assumptions. See Section \ref{sec:FluidMod} for  full details. 
We couple systems \ref{sys:NSE} or \ref{sys:3DNSE} to the advection-diffusion equation
\begin{equation}\label{eq:scalar-intro}
\begin{aligned}
& \partial_t g^\kappa_t + u_t \cdot \grad g^\kappa_t - \kappa \Delta g^\kappa_t = b \dot{\beta}_t \\
& g^\kappa_0 = g \, ,
\end{aligned}
\end{equation}
where $\beta_t$ is a Wiener process on another canonical probability space $(\Omega_\beta,\mathscr{F}^\beta,(\mathscr{F}^\beta_t),\P_{\beta})$. We denote the product measure $\P = \P_W\times \P_{\beta}$ our main probability measure on the associated product space $\Omega = \Omega_W\times\Omega_\beta$ with standard product sigma-algebra $\mathscr{F} = \mathscr{F}^W\otimes\mathscr{F}^\beta$ and filtration $\mathscr{F}_t= \mathscr{F}^W_t\otimes\mathscr{F}^\beta_t$ . Equation \eqref{eq:scalar-intro} has a $\P$ almost-sure, unique, $\mathscr{F}_t$-adapted weak solution for every initial $(u,g) \in \Hbf \times L^2$ (see Proposition \ref{prop:WPapp} below) and defines a Markov semigroup $P_t^\kappa$ on bounded, measurable observables $\varphi:\Hbf \times L^2 \to \R$ via
\begin{align}
P_t^\kappa \varphi(u,g) = 
\E_{(u,g)} \varphi(u_t, g_t^\kappa) := 
\EE\left[\varphi(u_t,g_t^\kappa) | (u_0,g_0^\kappa) = (u,g)\right]  \, .
\end{align}
    
\subsubsection{The cumulative Batchelor spectrum}\label{subsubsec:cumBatchSpec}
Recall, a probability measure $\mu$ on $\Hbf \times L^2$ is called \emph{stationary} for the Markov semigroup $P_t^\kappa$ if for all bounded, measurable $\varphi:\Hbf \times L^2 \to \R$ and all $t  > 0$,
\begin{align}
\int P_t^\kappa \varphi\,\dee\mu = \int \varphi \,\dee\mu. 
\end{align}
For observables $\varphi:\Hbf \times L^2 \to \R$, we frequently write $\E_\mu\varphi(u,g) := \int \varphi\, \dee \mu$. We will also say a probability measure is stationary for a given process if it is stationary for the corresponding Markov semi-group.

Due to the infinite-dimensionality of $\Hbf$, 
uniqueness of stationary measures for $(u_t)$ is in general a subtle question: 
on the domain $\T^2$ this was first proved in \cite{FM95} in the setting of completely nondegenerate noise, and has by now been established
even for highly degenerate noise \cite{HM06}.
In comparison, the homogeneous part of the evolution equation for $g_t^\kappa$ is linear, and so
uniqueness of stationary measures for $(u_t, g_t^\kappa)$ is relatively straightforward due to the contracting nature on $L^2$.
\begin{proposition} \label{prop:StatMes}
Assume $(u_t)$ admits a unique stationary probability measure $\mu$ on $\Hbf$. 
Then for all $\kappa > 0$, there exists a unique stationary measure $\mu^\kappa$ for $(u_t,g_t^\kappa)$ on $\Hbf \times L^2$. Moreover,
\begin{align}\label{eq:L2Balance}
2\kappa \E_{\mu^\kappa} \| \grad g\|_{L^2}^2 = \|b\|_{L^2}^2 =: \chi. 
\end{align}
\end{proposition}
Uniqueness of $\mu^\kappa$ is proved in Section \ref{subsec:unqiueness}, 
and has the following consequence by the Birkhoff ergodic theorem: for $\mu^\kappa$-typical
initial $(u,g) \in \Hbf \times L^2$ and any continuous observable $\phi \in L^1(\mu^\kappa)$
on $\Hbf \times L^2$, we have $\PP$-almost-surely\footnote{When $(u_t)$ is forced with nondegenerate noise, as we do in this paper, it is possible to promote the convergence in \eqref{eq:convBirkhoff} to \emph{all}
initial $(u, g) \in \Hbf \times L^2$ when one considers bounded, uniformly 
continuous observables $\varphi : \Hbf \times L^2 \to \R$. 
This follows from (1) the strong Feller property for
the Markov semigroup associated to $u_t$ (see \cite{FM95}) and (2) a small variation of the proof of Proposition \ref{prop:StatMes} given in Section \ref{subsec:unqiueness}. Details are omitted for brevity.}
\begin{align}\label{eq:convBirkhoff}
\lim_{T \to \infty} \frac{1}{T} \int_{0}^T \phi(u_t, g_t^\kappa)\,\dt = \int \phi\, \dee \mu^\kappa \, .
\end{align}
Equation \eqref{eq:L2Balance} follows from \eqref{eq:convBirkhoff} and It\^o's formula.
The convergence of time averages to ensemble averages in \eqref{eq:convBirkhoff} for typical initial data
$(u,g)$ confirms the statistically stationary setting described above: typical time-asymptotic behavior is captured by a unique stationary measure $\mu^\kappa$. 

Our main result is Batchelor's law on the \emph{cumulative} power spectrum (see Remark \ref{rmk:CPS}). Let $\Pi_{\leq N}$ denote the $L^2$ projection to the Fourier basis functions with frequency $\abs{k} \leq N$ (see Section \ref{sec:FluidMod}).
\begin{theorem}[Batchelor's law on the cumulative power spectrum] \label{thm:Batch}
There exists an $N_0$ (depending on $\nu$ but independent of $\kappa$) such that for all $\kappa \in (0,1)$ sufficiently small and $p \in [1,\infty)$, the following holds with implicit constants\footnote{We denote $f \lesssim_{p,q,...} h$ if there exists a constant $C>0$ depending on $p,q,...$ but independent of the other parameters of interest such that $f \leq Ch$ and $f \approx h$ when $f \lesssim h$ and $h \lesssim f$. For the entire paper, these implicit constants will \emph{never} depend on $\kappa, N$ or $t$.} \emph{independent of $\kappa$ and $N$}:
\begin{align}
\left(\E_{\mu^\kappa} \|\Pi_{\leq N} g\|_{L^2}^{2p}\right)^{1/p} & \approx_{p} \log N \quad\quad \textup{ for all }  N_0 \leq N \leq \kappa^{-1/2}.  \label{eq:PSIR}
\end{align}
Moreover, for all $s \in [0, \sigma-\tfrac{3d}{2}-1)$ and $\forall p \in [1,\infty)$, 
\begin{align}
  \left(\EE_{\mu^\kappa}\|g\|_{H^s}^{2p} \right)^{1/p} \lesssim_{p,s} \kappa^{-s} \abs{\log \kappa}. \label{ineq:preg} 
\end{align}
\end{theorem}

\begin{remark}
  Batchelor's law is often stated with a constant proportional to $\chi= 2\kappa \EE \norm{\grad g^\kappa}_{L^2}^2 = \|b\|_{L^2}^2$, as in \eqref{eq:Batch1}. 
A careful reading of our proof provides a simple estimate on the constants in \eqref{eq:PSIR} in terms of $b$, specifically, $\forall s > 0$, the following holds with implicit constants independent of $b$:
\begin{align}
\log{N}\frac{\chi^2}{\norm{b}_{H^s}^2} \lesssim_{p,s}  \left(\E_{\mu^\kappa} \|\Pi_{\leq N} g\|_{L^2}^{2p}\right)^{1/p} & \lesssim_{p,s} \norm{b}_{H^s}^2 \log N \quad\quad \textup{ for all }  N_0 \leq N \leq \kappa^{-1/2}, \label{ineq:Batch2}
\end{align}
and that it suffices to take 
\[
N_0 \approx_{s,\delta} \left( \frac{\norm{b}_{H^s}^{2/s}}{\chi^{1/s}} \right)^{1+\delta}
\]
for any $\delta > 0$ (the implicit constants in \eqref{ineq:Batch2} will also depend on $\delta$). 
The use of regularity is because mixing estimates require the source to have {\em some} positive regularity in order to get quantitative, $\kappa$-independent decay rates in negative Sobolev spaces. It is unclear if Batchelor's law as stated in Theorem \ref{thm:Batch} is still true if $b$ has no more than $L^2$ regularity.
If $b$ takes values in any $\Lambda_C = \set{b \in H^s:\, C^{-1}\norm{b}_{L^2} \leq \norm{b}_{H^s} \leq C\norm{b}_{L^2}}$, then both implicit constants in \eqref{eq:PSIR} are proportional to $\chi$ and $N_0$ depends only on $C$.
\end{remark}

Theorem \ref{thm:Batch} implies the following uniform in $\kappa$ estimate:
\begin{corollary}\label{cor:unifHdeltaEstBatch}
For all $s > 0$ and $p\geq 1$ the following holds
\begin{align}\label{eq:H-s-uniform-est}
\sup_{\kappa}\EE_{\mu^\kappa} \|g\|_{H^{-s}}^{2p} \lesssim_{p,s} 1.
\end{align}
\end{corollary}
\begin{proof}
Note that by Minkowski's inequality, (denoting $\Pi_N = \Pi_{\leq N} - \Pi_{\leq N/2}$ for $N \geq 2$ and $\Pi_1 = \Pi_{\leq 1}$), 
\begin{align}
\left(\EE
_{\mu^\kappa}\|g\|_{H^{-s}}^{2p}\right)^{1/p} \approx \left(\EE_{\mu^\kappa}\!\!\left(\sum_{j \geq 0} 2^{-2s j} \|\Pi_{2^{j}} g\|_{L^2}^{2}\right)^p \right)^{1/p} \lesssim \sum_{j \geq 0} 2^{-2s j} \left(\EE_{\mu^\kappa}\|\Pi_{2^j} g\|_{L^2}^{2p}\right)^{1/p} \,, 
%& \lesssim \sum_{j \in \N} 2^{-2s j} j,
\end{align}
which is bounded $\lesssim 1$ by \eqref{eq:PSIR}.
\end{proof}
\begin{remark}
Note that Theorem \ref{thm:Batch} (with $s=0$) implies the logarithmic divergence in $L^2$: 
\begin{align}\label{eq:L2-diverge}
\EE_{\mu^\kappa} \|g\|_{L^2}^2 \approx \abs{\log \kappa}. 
\end{align}
\end{remark}
\begin{remark}
  The cumulative power spectrum estimate in \eqref{eq:PSIR} comes in two parts.
 The lower bound is \emph{easier} to establish, and contains relatively little dynamical information: we show in Section \ref{subsec:lowerBoundBatch} that it follows from the fact that Lipschitz velocity fields cannot mix a scalar faster than exponential (provided that one has suitable exponential moment estimates to control large deviations). The upper bound on the other hand is much more difficult. It makes crucial use of the optimal, almost-sure uniform in $\kappa$ mixing estimates obtained in our recent work \cite{BBPS19II}, which in turn depends heavily on the results and methods of our earlier works on mixing and Lagrangian chaos \cite{BBPS18,BBPS19I}. Without these results, we would not be able to obtain an upper bound. See below in Section \ref{subsec:mechanismBatch} for more discussion on this sequence of works.
\end{remark}
\begin{remark}
The proof of the upper bound in \eqref{eq:PSIR} holds directly for all $N < \infty$. By itself, the estimates on the cumulative power spectrum in \eqref{eq:PSIR} are not precise enough to localize the dissipative range. Indeed, it is easy to check that \eqref{eq:PSIR} implies for all $\zeta > 0$,
\begin{align}
\left(\EE_{\mu^\kappa}\|\Pi_{\leq N} g\|_{L^2}^{2p}\right)^{1/p} & \approx_{\zeta,p} \log N \quad\quad \textup{ for all }  N_0 \leq N \leq \kappa^{-\zeta}.  \label{eq:PSIRfail}
\end{align}
However, \eqref{eq:L2Balance} provides the additional information needed to localize the dissipative scale to approximately $\kappa^{-1/2}$, as one should expect from parabolic regularity. Indeed, 
\begin{align}
\EE_{\mu^\kappa}\|\left(I - \Pi_{\leq N}\right)g\|_{L^2}^2  \leq \frac{1}{N^2}\EE_{\mu^\kappa}\|\grad g\|_{L^2}^2 = \frac{\chi}{2\kappa N^2} \, .
\end{align}
The estimate \eqref{ineq:preg} provides more precise high frequency moment control, albeit with a logarithmic deviation from \eqref{eq:L2Balance}.
\end{remark}

\begin{remark} \label{rmk:CPS}
The cumulative power spectrum estimate \eqref{eq:PSIR} is a little weaker than \eqref{eq:Batch1} or a dyadic-shell averaged version (i.e. $\EE\norm{\Pi_{2^j} g^\kappa}_{L^2}^2 \approx 1$).  
However, \emph{if} $\EE_{\mu^\kappa} \abs{\hat{g}(k)^2} \approx F(k)$ for $F(k)$ a monotone function, \emph{then} \eqref{eq:PSIR} implies that
$F(k) \approx \abs{k}^{-d}$ (and analogously for any dyadic-shell averaged version). In a dyadic-shell averaged version, the only kind of violations of Batchelor's law not ruled out by \eqref{eq:PSIR} are if some dyadic shells have too little mass and/or a sparse set (i.e. zero asymptotic density) of dyadic shells have have too much $L^2$ mass (but never more than $\log 2^j$). The pointwise version in \eqref{eq:Batch1} could also be violated if the $L^2$ mass was not distributed evenly enough in angle. % (as in, for example, the CAT map example in Section \ref{subsec:mechanismBatch}). 
This is the case, for example, for discrete-time pulsed-diffusion models of advection-diffusion by CAT maps (see e.g., \cite{IF19}).
Bridging this gap is the subject of a possible future line of research: see Section \ref{sec:Future} below. 
\end{remark} 

\subsubsection{Yaglom's law}\label{subsubsec:Yaglom}

Yaglom's law was predicted in 1949 in \cite{Yaglom49} for all passive scalar turbulence regimes, and consists of a reformulation of the constant scale-by-scale $L^2$ flux characteristic of anomalous dissipation. This is analogous to scaling laws for other turbulent systems, e.g., the Kolmogorov $-4/5$ law for 3D Navier-Stokes (see \cite{Frisch1995,Nie1999,BCZPSW18} and the references therein). 
In \cite{BBPS18}, we showed as a consequence of Lagrangian chaos that Yaglom's law holds over \emph{some} inertial range with an 
unspecified lower bound $\ell_\kappa$ with $\ell_\kappa \to 0$ as $\kappa \to 0$.
Theorem \ref{thm:Batch} allows us to 
%refine this and 
show that this law essentially holds over the entire inertial range $\ell \gtrsim \kappa^{1/2}$.
The proof is an easy application of the methods of \cite{BCZPSW18,BBPS18} and the a priori estimate \eqref{eq:L2-diverge} from Theorem \ref{thm:Batch} and so is omitted for brevity. 
\begin{corollary}[Yaglom's law over sharp inertial range]\label{cor:yaglomFinite}
Denote the finite increment for $h \in \R^d$,
\begin{align}
\delta_h g(x) = g(x+h) - g(x). 
\end{align}
Let $\ell_\kappa > 0$ be such that $\ell_\kappa^{-1} = o\!\left((\kappa\abs{\log \kappa})^{-1/2}\right)$ as $\kappa \to 0$. Then, 
\begin{align}
\lim_{\ell_I \to 0} \limsup_{\kappa \to 0} \!\!\sup_{\ell \in (\ell_\kappa,\ell_I)} \abs{\frac{1}{\ell}\EE_{\mu^\kappa} \!\fint_{\T^d} \fint_{\S^{d-1}}\! \abs{\delta_{\ell n} g}^2 \delta_{\ell n} u \cdot n\, \dee n \dx + \frac{2}{d}\chi} = 0. \label{ineq:Yag}
\end{align}
\end{corollary}

Note that by \eqref{eq:L2-diverge}, statistically stationary solutions are blowing up in $L^2$, 
and so \eqref{ineq:Yag} is only possible with the assistance of a large amount of cancellations.

\begin{remark}
  For 3D Navier-Stokes, the Kolmogorov $-4/5$ law together with statistical self-similarity formally predicts the Kolmogorov $-5/3$ power spectrum ($\abs{k}^{2}\EE \abs{\hat{u}(k)}^2 \approx \abs{k}^{-5/3}$).
 It is unclear at the moment how intermittency may or may not add corrections \cite{kaneda2003energy, AnselmetEtAl1984, Frisch1995, SreeniKail93,MW96}.  While passive scalar turbulence in the Batchelor regime is expected to display intermittency \cite{PS90, shraiman1998anomalous, HZ94, ott1988chaotic, MK99}, we nevertheless do not see intermittency corrections to Batchelor's prediction on the power spectrum, at least in the case where $\nu$ is fixed.
Heuristically, one might guess this from Yaglom's law \eqref{ineq:Yag}: the regularity of the velocity implies that \eqref{ineq:Yag} formally scales like the second order structure function $\EE_{\mu^\kappa}\|\delta_\ell g\|_{L^2}^2$ which formally scales as the power spectrum by the Wiener-Khinchin theorem. Hence, Yaglom's law suggests a certain rigidity to Batchelor-regime passive scalars that is not present in many other `turbulent' systems.
\end{remark}

\subsubsection{The vanishing diffusivity limit} \label{sec:ADOC}
For many turbulent systems consisting of a weakly damped system subjected to forcing, it is expected
to be able to pass to the zero-damping limit and obtain statistically stationary solutions for the zero-damping problem  \cite{Nazarenko11,FGHV16,MSE07}. The limiting regime is sometimes referred to as \emph{ideal turbulence}
\cite{eyink2018review}; in this limit, one expects weak solutions in very low regularity spaces with a non-vanishing scale-by-scale flux of conserved quantities through \emph{all} sufficiently small scales.
A prominent line of mathematical research in this direction is the work on Onsager's conjecture for the 3D Euler equations \cite{Eyink94,Isett18,DLS12,DLS13,eyink2018review,CET}, which seeks to identify low-regularity 
solutions to the Euler equations capable of dissipating kinetic energy as one expects of
ideal turbulent solutions to 3D Euler. Work on weak turbulence in dispersive equations can also be considered to be in a type of `ideal turbulence limit' \cite{FGH16,BGHS19}.

A consequence of our results on the Batchelor spectrum is a realization of the ideal turbulence program
for passive scalar turbulence. Indeed, in contrast with contemporary advances on Onsager criticality for 
3d Euler (capable only of generating specific solutions to 3d Euler with nonvanishing flux), in our setting
we are able to exhibit probability measures supported in low-regularity spaces, typical samples of which
exhibit the desired scale-by-scale flux across all sufficiently small scales.\footnote{See \cite{FGHV16,MSE07} for examples of simplified shell models where the inviscid limit to rough, weak solutions dissipating constant energy through the inertial range was successfully carried out in the statistically stationary regime.} 

For this, we consider weak-$*$ subsequential limits of the sequence $\{ \mu^\kappa\}$ of stationary measures for the 
$(u_t, g_t^\kappa)$ process as $\kappa \to 0$, yielding (possibly more than one) stationary measures $\mu^0$ for the
zero-diffusivity process $(u_t, g_t^0)$ governed by
\begin{align}\label{eq:kappa=0}
\begin{aligned}
	&\partial_t u_t + u_t \cdot \grad u_t + \grad p_t - \nu \Delta u_t = Q \dot{W}_t \, , \\
	&\partial_t g_t^0 + u_t\cdot \nabla g^0_t = b \dot{\beta}_t \, .
\end{aligned}
\end{align}
Existence of weak-$*$ limits in $H^{-s}, s > 0$ follows from Prokhorov's theorem and the
 $\kappa$-uniform moment estimates in Corollary \ref{cor:unifHdeltaEstBatch}; as a result, 
the limiting measures $\mu^0$ are supported on some low-regularity subspace of $H^- = \cap_{s > 0} H^{-s}$. 
See Section \ref{subsec:tightness} for more details. The following summarizes the 
basic properties of this construction.

\begin{theorem}[Vanishing dissipation limit] \label{thm:kappato0}
There exists a subsequence $\{\mu^{\kappa_n}\}$ and a limit measure $\mu^0$ on $\Hbf \times H^{-}$ such that for each $s >0$, $\mu^{\kappa_n}$ converges weakly to $\mu^0$ as a measure on $\Hbf \times H^{-s}$. Moreover, any limit point $\mu^0$ is a stationary measure for $(u_t,g_t^0)$ and the limit satisfies Batchelor's law over an infinite inertial range: for $N_0$ as in Theorem \ref{thm:Batch} there holds 
\begin{equation}\label{eq:infinite-Batchelor}
	\left(\E_{\mu^0} \|\Pi_{\leq N}g\|_{L^2}^{2p}\right)^{1/p} \approx_p \log{N} \quad \text{for all }N \geq N_0,
\end{equation}
and moreover, there is non-vanishing $L^2$ flux: 
\begin{equation}\label{eq:AD212}
\EE_{\mu^0}\brak{\Pi_{\leq N}(ug), \grad \Pi_{\leq N}g } = -\frac{1}{2} \|\Pi_{\leq N} b\|_{L^2}^2. 
\end{equation}
\end{theorem}
\begin{remark}
Note that the absence of the $\kappa \Delta$ dissipation in the transport equation \eqref{eq:kappa=0} makes the existence of non-trivial stationary solutions far from obvious. 
The Markov process $(u_t, g_t^0)$ has bad regularity properties on $\Hbf \times H^{-s}$ (it is not Feller) due to a lack of stability in the transport equation with respect
 to $\Hbf$ perturbations of the velocity field, and so a Krylov-Bogoliubov argument using the uniform mixing bound does not apply. See Section \ref{sec:restrictedFeller} for more discussion. 
\end{remark}

\begin{remark}
The non-vanishing flux \eqref{eq:AD212} is analogous to Yaglom's law \eqref{ineq:Yag} for the $\kappa = 0$ equation, in that both results say something about the the constancy of the $L^2$ flux in the $\kappa \to 0$ limit. However, as we show below in Theorem \ref{thm:Ocrit}, $\mu^0$ generic $g$ are not even integrable on $\T^d$, and so we are unsure how to pass the $\kappa \to 0$ the limit in Yaglom's law directly.
\end{remark}

\subsubsection{Irregularity of $\kappa = 0$ stationary statistics}\label{subsubsec:oCritSubsc}

As mentioned previously, the limiting stationary statistics associated to $\mu^0$ are very irregular. Indeed, the uniform bound \eqref{eq:H-s-uniform-est} used to extract the limit $\mu^0$ suggests that the measure is concentrated, at best, in $H^{-} = \cap_{s > 0} H^{-s}$, while the limiting version of Batchelor's law \eqref{eq:infinite-Batchelor} implies that 
\begin{equation}\label{eq:L^2-moment-blowup}
	\E_{\mu^0}\|g\|_{L^2}^{2} = \infty \, , 
\end{equation}
i.e., $\mu^0$ cannot give second moments to functions in $L^2$. The primary reason for this irregularity is that, in the absence of a Laplacian, stationary solutions to the zero diffusivity equation
\begin{equation}\label{eq:kappa0-transport}
	\partial_t g_t^0 + u_t\cdot\nabla g_t^0 = b\dot{\beta}_t
\end{equation}
must ``anomalously'' dissipate the input from the noise through the advection term. Since the velocity field $(u_t)$ is regular we cannot rely on roughness of the velocity field to dissipate and must rely solely on the mixing properties of $(u_t)$. Due to the regularity of $(u_t)$, one can prove that the $\kappa = 0$ transport equation \eqref{eq:kappa=0} conserves the $L^1$ norm of $g^0_t$, and therefore one should expect that $L^1$ solutions to \eqref{eq:kappa0-transport} cannot dissipate. In fact, we are able to confirm that $\mu^0$ assigns zero mass to $L^1$ (see Section \ref{subsec:L1} for proof): 
%This follows from the DiPerna/Lions theory \cite{DPL89} for the transport equation.

%In the DiPerna/Lions theory \cite{DPL89} for rough solutions of transport equations, an important role is played by the commutator
%\begin{equation}\label{eq:commutator-intro}
%	(u\cdot\nabla g)_{\ep} - u\cdot\nabla (g)_{\ep},
%\end{equation}
%where $(\cdot)_{\ep}$ denotes mollification to scale $\ep>0$.
%The equation \eqref{eq:kappa0-transport} cannot dissipate if \eqref{eq:commutator-intro} vanishes, and since the velocity field is regular, the only way that such a commutator would not vanish is if the associated scalar $(g_t)$ is sufficiently irregular. One can check that if the velocity field is Lipschitz and $(g_t)$ belongs to $L^1$, then \eqref{eq:commutator-intro} vanishes in $L^1$ as $\ep \to 0$. 
%As a consequence of this type of analysis, we obtain the following theorem.

\begin{theorem}[Limiting solutions are not $L^1$]\label{thm:Ocrit} Any limit measure $\mu^0$ from Theorem \ref{thm:kappato0} satisfies
\[
	\mu^0(\Hbf\times L^1)=0.
\]
\end{theorem}

\begin{remark}
It is important to note that contrary to \eqref{eq:L^2-moment-blowup}, which only states that $\mu^0$ can't have second moments on $L^2$, Theorem \ref{thm:Ocrit} is much stronger in that it implies that $\mu^0$ generic scalars $g$ are ``true'' distributions in the sense that they do not take values in any space of integrable functions, regardless of moments.
\end{remark}

\begin{remark}
In general, DiPerna/Lions theory (see, e.g., \cite{DPL89}) predicts that \eqref{eq:kappa=0} conserves the $L^p$ norm of $g_t^0$ if the velocity field is in the Sobolev space $W^{1,q}(\T^d)$ for $q=p/(p-1)$. In general, velocity fields $(u_t)$ which are not Lipschitz do not propagate $L^1$ and should not be considered as belonging to the ``Batchelor Regime''. See \cite{MR4029736} for an example of a continuous $W^{1,p}$ velocity field which does not propagate $L^1$ and also \cite{DEGI19} for a related example.
\end{remark}

Theorem \ref{thm:Ocrit} is based on understanding the formal $L^1$ conservation law of the inviscid equation \eqref{eq:kappa0-transport}.
However, another clear question is to determine how irregular $g$ must be to have a non-vanishing $L^2$ flux as in \eqref{eq:AD212}, i.e. studying violation of the $L^2$ conservation law.
This is analogous to the problem for weak solutions of the 3D Euler equations known as {\em Onsager's conjecture} which has received a significant amount of mathematical attention in recent years (see e.g. \cite{I18,DLS13,DS09} and the review \cite{BV19}). In the context of the 3D Euler equations, Onsager's conjecture states that weak solutions can dissipate energy when not in $C_{t,x}^{0,1/3-}$ and cannot dissipate energy if in $C_{t,x}^{0,1/3+}$. The easier direction, that weak solutions of 3D Euler conserve energy with sufficient regularity, was studied previously in \cite{Eyink94,CET,CCFS08}. Specifically, in \cite{CCFS08} it was shown that the Onsager-critical space (the space which divides dissipative from conservative) is the Besov space $L^3_t B^{1/3}_{3,\infty}$.

In the spirit of \cite{CCFS08}, we show below that the space $L^2_t B^0_{2,\infty}$ is Onsager-type critical for the passive scalar turbulence problem, where $B^0_{2,\infty}\subseteq H^-$ is the Besov space of tempered distributions $f\in H^-$ such that
\[
\sup_{N\in \set{2^j: j \in \N_\ast}}\|\Pi_{N}f\|_{L^2} <\infty,
\] 
 (recall $\Pi_N = \Pi_{\leq N} - \Pi_{\leq N/2}$ is projection onto the dyadic shell of frequencies of length between $N$ and $N/2$ and $\N_\ast = \set{0} \cup \N$).
Time-integrability will be connected to moments with respect to the stationary measure $\mu^0$ since for statistically stationary processes, expectations of time averages are exactly equal to ensemble averages. 
Just past this critical boundary is the space $B_{2,c}^0 \subset B_{2, \infty}^0$ of distributions such that 
\begin{align}
\limsup_{j \to \infty} \norm{\Pi_{2^j} f}_{L^2} = 0 \, , 
\end{align}
which possess just barely enough regularity to rule out $L^2$ flux (See Section \ref{subsec:OcritFuckery}). 

%Understanding when the $L^2$ flux vanishes reduces to essentially the same commutator as in the $L^1$ case above in \eqref{eq:commutator-intro}: 
%\begin{equation}\label{eq:L2fluxComInt}
%	\langle \Pi_{\leq N}(u\cdot \grad g),\Pi_{\leq N}g \rangle_{L^2} = 	\langle \Pi_{\leq N}(u\cdot \grad g) - u \cdot \grad \Pi_{\leq N} g,\Pi_{\leq N}g \rangle_{L^2}.
%\end{equation}
%Indeed, the projection $\Pi_{\leq N}$ is a particular type of mollifier where here $N$ plays the role of $\epsilon^{-1}$.  
%We show that this commutator vanishes as $N \to \infty$ in the subspace $B_{2,c}^0 \subset B_{2,\infty}^0$ of distributions such that 
%\begin{align}
%\limsup_{j \to \infty} \norm{\Pi_{2^j} f}_{L^2} = 0. 
%\end{align}
In order to quantify regularity in $B^0_{2, \infty}$, we introduce generalized fractional derivative norms, which play the same role that the modulus of continuity does to generalize H\"older regularity.
\begin{definition}\label{def:gen-mult}
We call a multiplier $M:[0,\infty) \to [1,\infty)$ \emph{$B_{2,c}^0$-suitable} if
\begin{itemize}
\item[(i)] $M$ is monotone increasing and $\lim_{k \to \infty} M(k) = \infty$;
\item[(ii)] $M$ is (globally) Lipschitz continuous and $\forall C > 0$, $\exists c > 0$ such that if  $C^{-1}\abs{\ell} \leq \abs{k} \leq C \abs{\ell}$, 
\begin{align}
\abs{M(k) - M(\ell)} \leq \frac{c}{\sqrt{1+|k|^2}} M(\ell) \abs{k-\ell}. 
\end{align}
\end{itemize}
\end{definition}
Given such an $M$, we define the generalized Besov norm 
\begin{align}
\norm{f}_{B_{2,\infty}^M} := \sup_{N \in \set{2^j: j \in \N_\ast}}M(N) \norm{ \Pi_N f}_{L^2}. 
\end{align}
We show below in Lemma \ref{lem:quantRegBesov} that $f\in B_{2,c}^0$ if and only if $\norm{f}_{B^M_{2, \infty}} < \infty$ for some $B_{2,c}^0$-suitable $M$.
We prove the following theorem by contradiction with \eqref{eq:AD212} in Section \ref{subsec:OcritFuckery}.  
\begin{theorem}[Onsager-type criticality of $L^2_t B_{2,\infty}^0$] \label{thm:OcritB2inf}
Let $\mu^0$ be a stationary measure for the $\kappa = 0$ limit process extended to $\Hbf \times H^{-1}$. Then for every $p > 2$ and all $B_{2,c}^0$-suitable $M$, 
\[
	\EE_{\mu^0}\|g\|_{B_{2,\infty}^M}^p = + \infty.
\] 
\end{theorem}

\begin{remark}
Note that we are not able to show that $\mu^0$ assigns zero mass to $B^{0}_{2,c}$ as we could for $L^1$, and instead can only show that moments with $p>2$ cannot be finite. This obstruction is related to the fact that moments in $\mu^0$ are related to time integrability of stationary solutions and the fact that the critical space is $L^2_tB^0_{2,\infty}$. It is unclear if it is possible for $\mu^0$ to assign positive measure to $\Hbf\times B^{0}_{2,c}$. 
\end{remark}

Using Theorems \ref{thm:OcritB2inf} and \ref{thm:Ocrit} we can deduce the following about the solutions 
studied in Theorem \ref{thm:Batch}. In particular, we show that moments of solutions diverge in certain norms, analogous to \eqref{eq:L2-diverge}.
\begin{corollary} \label{thm:div}
Let $\Phi$ on $[0,\infty)$ be a convex monotone function satisfying $\Phi(0) = 0$ and $\lim_{r\to \infty}\Phi(r)/r = \infty$ and let $M$ be a $B^0_{2,c}$-suitable multiplier. 
The unique stationary measure $\mu^\kappa$ on $\Hbf \times L^2$ satisfies the following for each $\delta > 0$
\begin{equation}\label{eq:blow-up-in-meas}
\lim_{\kappa \to 0} \E_{\mu_\kappa}\left(\int\Phi(|g|)\dx\right)^\delta  = \infty 
% \footnote{This is equivalent to saying that $\int\Phi(|g^\kappa|)\dx$ diverges in measure as $\kappa \to 0$, and naturually implies that all moments of $\int\Phi(|g^\kappa|)\dx$ also diverge as $\kappa \to 0$.}
\end{equation}
and for each $p>2$
\begin{equation}
\lim_{\kappa \to 0} \EE_{\mu^\kappa} \norm{g}_{B_{2,\infty}^M}^{p} = \infty.
\end{equation}
\end{corollary}

\subsection{Open problems in Batchelor-regime passive scalar turbulence}\label{sec:Future}

Theorems \ref{thm:Batch}--\ref{thm:div} provide a starting point for a mathematical understanding of Batchelor-regime passive scalar turbulence at fixed Reynolds number. However, there are many remaining open questions, all of which are potentially accessible in the near future using a combination of stochastic PDEs, harmonic analysis, and random dynamical systems. Let us briefly outline these here. 
\begin{itemize}[leftmargin=*]
\item \textbf{Intermittency} Certainly the most important set of open questions regard intermittency.
  Following discussions in e.g. \cite{Frisch1995,CS14} we can begin to study intermittency by looking at the flatness parameters 
\begin{align}
F_p(N) = \frac{\EE \norm{ \Pi_{N} g}^{2p}_{L^{2p}}}{\left(\EE \norm{\Pi_N g}_{L^2}^2\right)^{p}}. 
\end{align}
A non-intermittent field would satisfy $F_p(N) \leq C(p)$ as $N \lesssim \kappa^{-1/2}$, $N \to \infty$, $\kappa \to 0$.
An example of such a field is white noise. At the opposite extreme is a maximally intermittent random field consisting of single Dirac delta function placed with uniform probability on $\T^d$, which satisfies $F_p(N) \approx_p N^{p-1}$.
Passive scalar turbulence is expected to be intermittent \cite{PS90, shraiman1998anomalous, HZ94, ott1988chaotic}.
A major step in our understanding would be to provide an analytic derivation of powers $\zeta(p)$ such that $F_p(N) \approx N^{\zeta(p)}$ for $N \to \infty$, $\kappa \to 0$, if such powers exist. 

A helpful intermediate step might be to consider the power spectrum of a discrete-time pulsed diffusion model 
\cite{IF19}, e.g., 
%It is instructive to see that the CAT map \eqref{eq:CAT} example is in fact \emph{not} intermittent, in contrast to expectations regarding passive scalar turbulence. RDS toy systems extending \eqref{eq:CAT}, such as the CAT map replaced with the
using a randomly-driven Chirikov standard map to model advection\footnote{Note that the presence of noise makes this tractable, unlike the deterministic case which is notoriously difficult \cite{gorodetski2012stochastic,duarte1994plenty}. } \cite{blumenthal2017lyapunov}. %, might be able to shed some further light on intermittency effects. 

\item \textbf{How universal is ``universal''?} Another set of important problems is to study how widely applicable the Batchelor spectrum, and other tenets of the theory, such as uniqueness of stationary measures, Lagrangian chaos etc, to different and more realistic settings. 
\begin{itemize}[leftmargin=*]
\item Problems on $\T^d$ with body forcing are far removed from any real physical applications. Extending existing theories to include boundaries (exterior or interior domains) and replacing body forces with boundary driving are probably the most physically important extensions. Even relatively basic questions, such as uniqueness of stationary measures of the Navier-Stokes equations, are to our knowledge, quite challenging and still open for most questions of this type. See \cite{Shirikyan2018} for some progress in this direction. A related direction is to study spatially homogeneous solutions on $\mathbb R^d$.

\item Even on $\T^d$ with stochastic forcing, in the case of the Navier-Stokes equations, our results on Lagrangian chaos and scalar mixing \cite{BBPS18,BBPS19I,BBPS19II} cannot handle $C^\infty_x$ forcing yet. It is natural to seek to extend this to $C^\infty_x$ forcing and further to include non-white-in-time forcing such as OU tower forcing (see Section \ref{sec:CktCx} and \cite{BBPS19I})  and the class of bounded forcing studied in \cite{KNS18,KNS19,JNPS19}. Note that the hypoellipticity theory of Hairer and Mattingly \cite{HM11}  applies to the one-point Lagrangian flow ($\dot x_t = u_t(x_t)$) generated by 2D Navier-Stokes with OU tower forcing (and so the $(u_t,Z_t,x_t)$ Markov process has a unique stationary measure), however, our Lagrangian chaos results require strong Feller in order to use the version of the Furstenberg criterion in \cite{BBPS18}.
\end{itemize}
  
\item \textbf{Sharper regularity estimates and structure function renormalization.} Even just concerning basic questions related to Batchelor's law \eqref{eq:PSIR}, there are still remaining questions.
\begin{itemize}[leftmargin=*]
\item There are three basic levels of precision when discussing the power spectrum. After the cumulative spectrum, the next most difficult is a dyadic shell-by-shell estimate, which so far remains unaddressed by the results in this paper. The next most difficult after that is the pointwise estimate
\begin{align}
\EE_{\mu^\kappa} \abs{\hat{g}(k)}^2 \approx \abs{k}^{-d} \quad N_0 < \abs{k} < \kappa^{-1/2}. \label{ineq:ptwiseSpec}
\end{align}
In fact, one can even try to search for an estimate of the type (see discussion in the physics literature e.g. \cite{DSY10}), where of course we mean that the error is uniformly controlled in $\kappa$, 
\begin{align}
\EE_{\mu^\kappa} \abs{\hat{g}(k)}^2 = C_B\abs{k}^{-d} + o_{k \to \infty}(\abs{k}^{-d}) \quad N_0 < \abs{k} < \kappa^{-1/2}. 
\end{align}
These three basic levels are not equivalent. Pulsed-diffusion models based on discrete-time random dynamical systems might be able to shed some light on the subtle differences between these spectral characterizations. 

%RDS systems such as hyperbolic toral automorphisms drawn from random distributions might be able to shed some light on the subtle differences between these spectral characterizations.  

\item It is an interesting and subtle question to determine if the limiting solutions we obtain in Theorem \ref{thm:Batch} are \emph{exactly} in the Onsager critical space $L^2_t B^{0}_{2,\infty}$, or more to the point, whether the $\kappa > 0$ approximations are uniformly bounded in this space, that is,
\begin{align}
\EE_{\mu^\kappa} \sup_{N \in 2^{\mathbb N}} \norm{\Pi_N g}^2_{L^2} \lesssim 1.
\end{align}
%This is almost the same as the shell-by-shell Batchelor spectrum as in \eqref{eq:CATBatch}, but the $\EE$ and $\sup_{N \in 2^{\mathbb N} }$ are reversed.

\item Batchelor's law should concern the second order structure function $\EE_{\mu^\kappa}\norm{\delta_\ell g}_{L^2}^2$. However, as the scalar does not even remain a locally integrable function as $\kappa \to 0$, it is hard to make sense of exactly how the second order structure function will behave in and at the limit $\kappa \to 0$; instead, this may require a suitable renormalization. Similarly, Yaglom's law on the $L^2$ flux \eqref{ineq:Yag} is hard to make sense of rigorously at the $\kappa = 0$ limit. 
 \end{itemize} 
  
\end{itemize}

\section{Preliminaries} \label{sec:Prelim}
% !TEX root = master.tex

\subsection{Fluid models} \label{sec:FluidMod}

Let $d =2$ or $3$. We fix a real Fourier basis of 
of ${\bf L}^2 = \{ u \in L^2(\T^d, \R^d) : \int u \dee x = 0, \Div u = 0\}$ as follows:
for $m = (k, i) \in \K:= \Z_0^d \times \{ 1, \cdots, d-1\}$, we set
\[
e_m(x) = \begin{cases}
c_d \gamma_k^i \sin(k \cdot x) & k \in \Z_+^d \\
c_d \gamma_k^i \cos(k \cdot x) & k \in \Z_-^d \, .
\end{cases}
\]
Here, $\Z_0^d := \Z^d \setminus \{ 0 \}$, and $\Z_0^d = \Z_+^d \cup \Z_-^d$ is the partition
defined by $\Z_+^d = \{ k = (k^{(1)}, \cdots, k^{(d)}) \in \Z_0^d : k^{(d)} > 0 \} \cup \{ k \in \Z_0^d : k^{(1)} > 0, k^{(d)} = 0 \}$
and $\Z_-^d = - \Z_+^d$. For each $k \in \Z_0^d$, we have fixed a set $\{ \gamma_k^i\}_{i =1}^{d-1}$
of orthonormal vectors spanning the complement of the line spanned by $k \in \R^d$; these
are assumed to satisfy $\gamma_{-k}^i = - \gamma_k^i$. The coefficients $c_d > 0$ are normalization constants. Note that if $d = 2$, $\gamma_k = \gamma_k^1$ spans the perpendicular to $k$,
and may therefore be taken to be $\gamma_k = k^\perp / |k| , k^\perp := (k^{(2)}, - k^{(1)})$.

In terms of this Fourier basis, we consider the white-in-time, spatially Sobolev stochastic forcing
term
\begin{equation}\label{eq:noise-def}
Q W_t := \sum_{m \in \K} q_m e_m(x) W_t^{m} \, , 
\end{equation}
where $W_t^{m}, m \in \K$ are a family of independent standard one-dimensional Wiener processes with respect
to the cannonical stochastic basis $(\Omega_W, \mathscr F^W, (\mathscr F_t^W), \P_W)$.
The following decay and non-degeneracy assumption is made throughout this and our previous works \cite{BBPS18, BBPS19I, BBPS19II}.
\begin{assumption}\label{ass:strong}
	There exists $\alpha > \frac{5 d}{2}$ such that for all $m = (k, i) \in \K$, we have
	\[
	q_m \approx \frac{1}{|k|^\alpha} \, .
	\]  
\end{assumption}
The state space for our fluid velocity fields 
is 
\[
\Hbf := \left\{ u \in H^\sigma(\T^d, \R^d) : \int u \,\dee x = 0,\quad \Div u = 0 \right\} \, , 
\]
where $\sigma \in (\alpha - 2 (d - 1), \alpha - \frac{d}{2})$. Note that by our choice
of $\alpha$, we have $\sigma > \frac{d}{2} + 3$, so $\Hbf \hookrightarrow C^3$. 
We will write the Navier-Stokes system as an abstract evolution equation on $\Hbf$ by 
\begin{equation}\label{eq:NS-Abstract}
	\partial_t u + B(u,u) + Au = Q\dot{W} = \sum_{m\in\mathbb{K}} q_m e_m \dot{W}^m \, ,  
  \end{equation}
where 
\begin{align}
B(u,v) &= \left(\Id - \grad (-\Delta)^{-1} \grad \cdot \right)\grad \cdot (u \otimes v) \\
Au &= \begin{cases} -\nu \Delta u \quad&  \textup{ if } d=2 \\
-\nu' \Delta u + \nu\Delta^2 u \quad& \textup{ if } d=3. 
\end{cases}
\end{align}
The $(u_t)$ process with initial data $u$ is defined as the solution to \eqref{eq:NS-Abstract} in the mild sense \cite{KS,DPZ96}:
\begin{align}
u_t = e^{-tA}u - \int_0^t e^{-(t-s)A} B(u_s,u_s) ds + \int_0^t e^{-(t-s)A} Q \dee W(s) \, ,  \label{eq:Mild} 
\end{align}
where the above identity holds $\P_W$ almost surely for all $t>0$. We have the following well-posedness theorem:
\begin{proposition}[\cite{KS,DPZ96}] \label{prop:WPapp}
Let $d = 2$ or $3$. Under Assumption \ref{ass:strong}, for all initial $u \in \Hbf \cap \Hbf^{\sigma'}$ with $\sigma' < \alpha-\frac{d}{2}$ and all $T> 0, p \geq 1$, there exists a $\P_W$-a.s. unique solution $(u_t)$ to \eqref{eq:Mild} which is $\mathscr{F}_t^W$-adapted, and belongs to $L^p(\Omega_{W};C([0,T];\Hbf \cap \Hbf^{\sigma'})) \cap L^2(\Omega_W;L^2(0,T;\Hbf^{\sigma'+(d-1)}))$. Additionally, for all $p \geq 1$ and $0 \leq \sigma' < \sigma'' < \alpha - \frac{d}{2}$,  
\begin{align}
\EE_W \sup_{t \in [0,T]} \norm{u_t}_{\Hbf^{\sigma'}}^p & \lesssim_{T,p,\sigma'} 1 + \norm{u}_{\Hbf \cap \Hbf^{\sigma'}}^p \\
\EE_W \int_0^T \norm{u_s}_{\Hbf^{\sigma' + (d-1)}}^2 ds & \lesssim_{T,\delta} 1 + \norm{u}^2_{\Hbf^{\sigma'}} \\ 
\EE_W \sup_{t \in [0,T]} \left(t^{\frac{\sigma''-\sigma'}{2(d-1)}} \norm{u_t}_{\Hbf^{\sigma''}}\right)^p &\lesssim_{p,T,\sigma',\sigma''} 1 + \norm{u}^p_{\Hbf^{\sigma'}}. \label{ineq:locRegu}
\end{align}
\end{proposition}

\medskip

\begin{proposition}\label{prop:uniqueStatMeasNSE}
Under Assumption \ref{ass:strong}, the process $(u_t)$ solving
\ref{eq:NS-Abstract} admits a unique stationary measure $\mu$. 
\end{proposition}
Proposition \ref{prop:uniqueStatMeasNSE} was first proved for $d = 2$ in 
\cite{FM95} under Assumption \ref{ass:strong}; 
the generalization to $d = 3$ is a straightforward extension. 

\begin{remark}
Uniqueness of stationary measures is known for Navier-Stokes under much weaker 
nondegeneracy conditions than Assumption \ref{ass:strong}, e.g., the \emph{truly hypoelliptic} setting of \cite{HM06}
with $d = 2$ which only requires to force modes $|m|_\infty \leq 1$. However, Assumption
\ref{ass:strong} is necessary for the Lagrangian chaos and scalar mixing results
in \cite{BBPS18, BBPS19I, BBPS19II} because our methods require, for now,
strong Feller regularity of the semigroup $t \mapsto u_t$. See Remark 2.6 in 
\cite{BBPS18} and Remark 2.19 in \cite{BBPS19I} for more discussion. 
\end{remark}

Because we require Assumption \ref{ass:strong}, 
 we are not aware of how to extend our results to fluid models solving 
the Navier-Stokes equations which are spatially $C^\infty$ or differentiable in time. 
However, we are able to apply our results to a class of finite-dimensional
fluid models for which solutions are $C^k_t C_x^\infty$. This is the subject of the following 
short section.

\subsection{$C^k_t C^\infty_x$ fluid models governed by finite dimensional SDE} \label{sec:CktCx}

If the fluid evolves according to a finite-dimensional SDE, then the methods involved are significantly simpler at a technical level and the strength of H\"ormander's theorem \cite{Hormander85} allows us to impose much weaker conditions on the noise models we consider. Consequently, we can produce fluid models which have better spatial and time regularity. 

Below, for $Q W_t$ given by \eqref{eq:noise-def}, we define $\mathcal K_0 = \{ m \in \K : q_m \neq 0\}$. 
For $m = (k, i) \in \K = \Z_0^d \times \{ 1,\cdots, d-1\}$, we define $|m|_\infty = \max_j |k^{(j)}|$. 

For the following finite dimensional stochastic fluids models we will make the the following less restrictive assumption to Assumption \ref{ass:strong}.
\begin{assumption}\label{ass:weak}
Assume $m \in \mathcal K_0$ if $|m|_\infty \leq 2$. 
\end{assumption}

For $\mathcal K \subset \K$ we define $\Hbf_{\mathcal K}$ to be the linear span of the 
Fourier modes $\{ e_m \}_{m \in \mathcal K}$. One model we consider is the Stokes system
on $\T^d, d =2,3$, prescribing the time evolution $(u_t)$ on the state space 
$\Hbf_{\mathcal K_0}$ for fixed initial $u_0 \in \Hbf_{\mathcal K_0}$ by 
\begin{equation}\label{eq:stokes}
\begin{aligned}
&\partial_t u_t = - \nabla p_t + \Delta u_t + Q \dot W_t \\
&\Div u_t = 0.
\end{aligned}
\end{equation}
When $\mathcal K_0$ is finite, \eqref{eq:stokes} is a finite-dimensional SDE on $\Hbf_{\mathcal K_0} \cong \R^{| \mathcal K_0|}$; indeed, it is essentially a product of
independent Ornstein-Uhlenbeck processes on $\R^{| \mathcal K_0|}$, and in particular
an elliptic diffusion on $\R^{| \mathcal K_0|}$. 

For $N \geq 1$ we define
$\Hbf_N \subset \Hbf$ to be the linear span of the $e_m$ with $|m|_\infty \leq N$. 
Define $\Pi_{\leq N} : \Hbf \to \Hbf_N$ to be the orthogonal projection.
Another model we consider for $(u_t)$ is the Galerkin-Navier-Stokes system, defined for fixed
$u_0 \in \Hbf_N$ by 
\begin{equation}\label{eq:galerkin}
\begin{aligned}
&\partial_t u_t = - \Pi_{\leq N} ( \grad p_t + u_t \cdot \grad u_t) + \nu \Delta u_t + Q \dot W_t \\
&\Div u_t = 0 \, ,
\end{aligned}
\end{equation}
where implicitly we assume that $q_m \neq 0$ only if $|m|_\infty \leq N$. 
As with the Stokes system, Galerkin-Navier-Stokes is an SDE on the finite-dimensional space
$\Hbf_N$; under Assumption \ref{ass:weak}, equation \eqref{eq:galerkin} is known to 
satisfy the parabolic H\"ormander condition \cite{E2001-lg, Romito2004-rc}, 
and so results in a hypoelliptic diffusion. 

Standard finite-dimensional stochastic analysis \cite{kunita1996stochastic} applies to each of \eqref{eq:stokes}, \eqref{eq:galerkin}, yielding velocity field processes $(u_t)$ which are $C^{\frac12 -}$ in time and spatially $C^\infty$. In addition to these, our methods also apply to a class of models including those which vary $C^k$ in time for any fixed $k \geq 1$. These models are effectively driven by the projection of a coupled system of Ornstein-Uhlenbeck processes. More precisely: fix $2 \leq N \leq M$ and let 
$\mathcal A : \Hbf_M \times \Hbf_M$ be diagonalizable with strictly positive spectrum. 
Let 
\[
\Gamma W_t = \sum_{|m|_\infty \leq M} \Gamma_m e_m W_t^m
\] 
and let $X : \Hbf_N \times \Hbf_N \to \Hbf_N$ be a bilinear mapping with 
$u \cdot X(u,u) = 0$ and $X(e_m, e_m) = 0$ for all $|m|_\infty \leq N$. 
We consider the following \emph{generalized Galerkin-Navier-Stokes system with OU tower noise}, defined by
\begin{equation}\label{eq:generalizedGNSE}
\begin{aligned}
&\partial_t u_t = - X(u_t, u_t) + \nu \Delta u_t + Q Z_t \\
&\partial_t Z_t = - \mathcal A Z_t + \Gamma W_t.
\end{aligned}
\end{equation}
Here, the noise term applied to the $(u_t)$ process is 
$Q Z_t := \sum_{|m|_\infty \leq N} q_m Z_t^m e_m(x)$, where $t \mapsto Z_t^m\in \R$ are the Fourier coefficients of $Z_t$, i.e.,  $Z_t(x)= \sum_m Z_t^m e_m(x)$. 
Note that $(u_t)$ is not a Markov process, but $(Z_t)$ alone and $(u_t, Z_t)$ are. 
The appropriate nondegeneracy assumption in this setting is as follows:
\begin{assumption}\label{ass:genOU}
Assume that the coefficients $\{ q_m\}$ satisfy Assumption \ref{ass:weak}, and additionally 
that the parabolic H\"{o}rmander condition holds for the $(u_t, Z_t)$ process on $\Hbf_N \times \Hbf_M$.
\end{assumption}

\begin{theorem}\label{thm:finDim}
All of the main results in Section \ref{subsec:main-results} hold whenever $(u_t)$ evolves according to 
\eqref{eq:stokes} or \eqref{eq:galerkin} under Assumption \ref{ass:weak}, 
or \eqref{eq:generalizedGNSE} under Assumption \ref{ass:genOU}.
\end{theorem}

\begin{remark}[see Remark 1.10 in \cite{BBPS19I}]\label{rmk:CkTCinfX}
Consider the following example of a system in the setting of \eqref{eq:generalizedGNSE}. 
Fix $n \geq 1$ and consider the model
\begin{align} \label{eq:smoothInTime}
u_t(x) = \sum_{|m|_\infty \leq 2} u_t^m e_m(x) \, , 
\end{align}
where the coefficients $u_t^m$ evolve according to
\begin{align*}
\partial_t {u}^m_t & = - {u}^m_t + Z^{m,0}_t \\
\partial_t Z^{m,\ell} & = - Z_t^{m,\ell} + Z_t^{m,\ell+1} \quad 1 \leq \ell \leq n -1 \,\\
\partial_t Z^{m,n} & = - Z_t^{m,n} +\dot{W}^m_t.
\end{align*}
Up to re-indexing, this fits into the framework of \eqref{eq:generalizedGNSE} with $X \equiv 0$.
The parabolic H\"{o}rmander condition for $(u_t, Z_t)$ is satisfied, and so 
Theorem \ref{thm:finDim} holds for $(u_t)$ as above. Notably, solutions $(u_t)$ to \eqref{eq:smoothInTime} are $C^{n+1}$-differentiable in time and smooth (indeed, analytic) in 
space. The authors hope this serves as an indication that although the
 methods in this paper rely strongly on the stochastic framework, they are not inherently restricted 
 to the rough time regularity of white-in-time noise. 
\end{remark}

Going forward, we will assume $(u_t)$ solves \eqref{eq:NS-Abstract} with Assumption \ref{ass:strong} for remainder of the paper. The application of these arguments to the finite dimensional models \eqref{eq:stokes}, \eqref{eq:galerkin} and \eqref{eq:generalizedGNSE} is straightforward and omitted for brevity.

\subsection{Lyapunov functions and exponential estimates}

The results in our previous series of papers
\cite{BBPS19I, BBPS19II} require the use of the family of Lyapunov functions
\begin{equation}\label{eq:V-def}
V(u) = V_{\beta, \eta}(u) := (1 + \| u \|_{\Hbf}^2)^\beta e^{\eta \| u \|_{\Wbf}^2}
\end{equation}
where $\beta \geq 0, \eta > 0$. Here, for velocity fields $u$ we define
\[
\| u \|_\Wbf :=
\begin{cases}
\| \mathrm{curl}\, u \|_{L^2} & d = 2 \\
\| u \|_{L^2} & d = 3 \, .
\end{cases}
\]

Below, we formulate a drift condition for the family $V = V_{\beta, \eta}$, ensuring that trajectories
of the $(u_t)$ process frequently visit the sublevel sets of $V$. 
Define $\eta_* = \nu / \mathcal Q$, where 
\[
\mathcal Q = 64
\begin{cases}
 \sup_{m \in \mathbb K} \abs{k}\abs{q_m} & d = 2 \\
 \sup_{m \in \mathbb K} \abs{q_m} & d = 3\, .
\end{cases}
\]

\begin{lemma}[Lemma 3.7 in \cite{BBPS19I}] \label{lem:TwistBd} Let $(u_t)$ be a solution to the stochastic Navier Stokes equations \eqref{eq:NS-Abstract} with initial data $u\in\Hbf$. 
For all $0\leq \gamma < \nu / 8$, $r\in (0,3)$, $C_0 \geq 0$, and $V(u) = V_{\beta,\eta}$ where $\beta \geq 0$ and $0 < e^{\gamma T}\eta < \eta_*$, there exist constants $c = 
c(\gamma,r,C_0,\beta,\eta) > 0, C = C(\gamma,r,C_0,\beta,\eta) \geq 1$ such that the following estimate holds for any $T > 0$: 
\begin{equation}\label{eq:Twistbd}
\EE_W\exp\left(C_0 \int_0^T\norm{u_s}_{\Hbf^r}\ds\right)\sup_{0\leq t\leq T}V^{e^{\gamma t}}(u_t) \leq C e^{c T} V(u).
\end{equation}

\end{lemma}

\begin{remark} \label{rmk:SuperL}
To connect \eqref{eq:Twistbd} with more standard drift conditions given in, e.g., \cite{meyn2012markov}: write $P_1$ for the Markov semi-group for the Navier-Stokes process $(u_t)$ and apply Jensen's inequality to \eqref{eq:Twistbd} to deduce that $\exists C_L > 0$ such that $P_1 V \leq  (e^{C_L}V)^{e^{-\gamma}}$.
In particular, we have the following drift condition: 
\begin{align}\label{eq:superLyapDriftCond}
\forall \delta > 0 \, , \exists C_\delta > 0 \quad \text{ such that } \quad P_1V \leq \delta V + C_{\delta} \, .
\end{align}
Note that the contraction constant $\delta > 0$ above can be made \emph{arbitrarily small}.
Iterating this bound, it is straightforward to show that $\forall \lambda > 0$, $\exists K_\lambda$ such that for all $t  > 0$ there holds
\begin{align}
P_t V \leq e^{-\lambda t}V + K_\lambda. \label{ineq:supLDct}
\end{align}
\end{remark}

We will also need the following basic stability estimate for solutions to the Navier-Stokes equations.
\begin{lemma}\label{lem:uStablEst}
For all $u,u' \in \Hbf$ and let $(u_t)$ and $(u'_t)$ be the corresponding solutions with the same noise-path $\omega$.
Then $\forall p \in [1,\infty)$, there exists a deterministic $K > 0$ such that following stability estimate holds for $V(u)$ as in \eqref{eq:V-def}, 
\begin{align}
\EE \norm{u_t-u'_t}_{\Hbf}^p \lesssim_p  e^{pKt}\left(V(u)^p + V(u')^p\right)\norm{u-u'}_{\Hbf}^p.
\end{align}
\end{lemma}
\begin{proof}
Define $w_t = u_t - u'_t$, 
\begin{align}
\partial_t w_t + B(u_t,w_t) + B(w_t,u_t')  + Aw_t =0. 
\end{align}
The proof proceeds as a variation of e.g. [Lemma 3.10; \cite{BBPS19I}].
Analogous to the calculations therein we have for some $C>0$ and some $r\in (\frac{d}{2}+1,3)$
\begin{align}
\norm{w_t}_{\Wbf} \leq \exp\left(C\int_0^t \norm{u_\tau}_{H^r}  + \norm{u_\tau'}_{H^r} \dee\tau \right) \norm{u-u'}_{\Wbf}
\end{align}
and for some $q,C > 0$ (possibly a different $C$),
\begin{align}
\norm{w_t}_{\Hbf}^2 & \lesssim \exp\left(C\int_0^t \norm{u_\tau}_{H^r}  + \norm{u_\tau'}_{H^r} \dee\tau \right) \left(\sup_{0< s < t}(\norm{u_\tau}^q_{\Hbf} + \norm{u_\tau'}^q_{\Hbf})
\right)\norm{u-u'}_{\Hbf}^2.
\end{align}
The result then follows from Lemma \ref{lem:TwistBd}. 
\end{proof}

\subsection{Uniform mixing and enhanced dissipation}

In this section we will summarize the results of \cite{BBPS18, BBPS19I, BBPS19II}
which are used to prove Batchelor's law (Theorem \ref{thm:Batch}) and the other results of this paper. Throughout, assume $(u_t)$ solves \eqref{eq:NS-Abstract} and that Assumption \ref{ass:strong} holds. Consider the advection-diffusion equation with diffusivity $0 \leq \kappa \ll 1$, 
\begin{align}
   \partial_t \bar g_t + u_t \cdot \grad \bar g_t = \kappa \Delta \bar g_t \label{def:AD} 
 \end{align}
for fixed initial $\bar g_0 = g  \in L^2, \int g \,\dee x = 0$. This defines 
the following (random) two-parameter semigroup of linear operators on $L^2$: 
for $0 \leq s \leq t$, $\omega \in \Omega_W$ and initial $u = u_0 \in \Hbf$, define
$S^\kappa_{t, s}(\omega, u) : L^2 \to L^2$ the two time solution operator, satisfying $\bar g_t = S^\kappa_{t,s}(\omega,u) \bar g_{s}$. When starting from $s = 0$ we will write 
$S^\kappa_t(\omega, u) = S^\kappa_{t, 0}(\omega, u)$. Note that $S_0^\kappa(\omega, u) = I$,
the identity on $L^2$. Note as well the following cocycle property: for any $s, t > 0$, we have
\[
S^\kappa_{t + s}(\omega, u) = S^\kappa_s(\theta_t \omega, u_t)S^\kappa_t (\omega, u) \,,
\] 
where, $\theta_t : \Omega_W \to \Omega_W$ is the standard time-shift on Wiener space defined for each $\omega \in \Omega_W$ by
\[
	\theta_t\omega(\cdot) := \omega(\cdot+t) - \omega(t).
\]
This implies that we can write the two parameter semi-group $S_{t,s}(\omega,u)$ as
\begin{equation}\label{eq:SkewS}
	S_{t,s}^\kappa(\omega,u) = S^\kappa_{t}(\theta_s\omega,u_s).
\end{equation}
Note that for any $p \geq 1$, $S^\kappa_t$ extends in a natural way to a defined
mapping on $L^p$, with range contained in $L^p$. The same is true of Sobolev spaces $H^s$, $s > 0$. 

Several of the important results in this section involve random constants of the form $D : \Omega_W \times \Hbf \to \R_{\geq 1}$,
the $\P$-law of which are controlled in terms of $V(u)$. Since random constants of this type appear many times in this paper, we introduce the following definition.

\begin{definition}\label{defn:randConst}
Let $D : \Omega_W \times \Hbf \to \R_{\geq 1}$ be measurable. We say that $D$ has
$V$ \emph{bounded} $p$-\emph{th moment} if $\exists \beta \geq 0$ such that 
$\forall 0 < \eta < \eta_* $, we have for $V = V_{\beta, \eta}(u)$
\[
\E D^p(\cdot, u) \lesssim_\eta V(u).
\]
\end{definition}

The following Lemma is very important since it provides us some control on the fluctuations of the mixing time in terms of the Lyapunov function $V(u)$.
\begin{lemma} \label{lem:RandCon}
Let $D : \Omega_W \times \Hbf \to [1,\infty)$ have $V$-bounded $p$-th moment. 
Then, $\forall \lambda >0 $, $\exists K_\lambda > 0$ such that
\begin{align}\label{ineq:Dbd}
\EE D^p(\theta_t \omega, u_t) \lesssim_{p,\lambda} e^{-\lambda t}V(u) + K_\lambda.
\end{align}
\end{lemma}
\begin{proof}
Using, in sequence,
the tower property of conditional expectation, the fact that $u_t$ is $\mathscr{F}_t^W$ measurable and the increment $\theta_t\omega$ is independent of $\mathscr{F}_t^W$, and 
equations \eqref{ineq:Dbd} and \eqref{ineq:supLDct}, we have: 
\begin{align}
\EE D^p(\theta_t \omega, u_t) & = \EE\left(\EE \left[D^p(\theta_t \omega, u_t)| \mathscr{F}_t^W\right]\right)
 \lesssim \EE V(u_t)  \lesssim e^{-\lambda t} V(u) + K_\lambda \, . \qedhere
\end{align}
\end{proof}

The following Lemma is a useful corollary of \eqref{eq:Twistbd}. 

\begin{lemma}\label{lem:twist-cantelli}
For all $p \in [1,\infty)$, there exists a (deterministic) $C_0 >0$ and a random constant $D_0 : \Hbf \times \Omega_W \to \R_{\geq 1}$ with $V$-bounded $p$-th moment such that
\[
	\exp\left( \int_0^t \norm{\grad u_\tau}_{L^\infty} d\tau \right) \leq D_0(\omega, u) e^{C_0 t} \, .
\]
\end{lemma}
\begin{proof}
We provide the proof when $p = 1$; other values of $p$ require straightforward adjustments.

Set $V = V_{0, \eta}$, where $\eta \in (0,\eta_*)$ is arbitrary. 
To start, note that by Lemma \ref{lem:TwistBd} and Chebyshev's inequality,
 $\exists c = c(\eta) > 0$ such that 
\begin{align}\label{eq:penguin}
	\P\left(  \exp\left( 2 \int_0^t\|\nabla u_s\|_{L^\infty}\ds\right) > V(u)e^{4ct}\right) \leq 
	\frac{\E  \exp\left(2 \int_0^t\|\nabla u_s\|_{L^\infty}\ds\right) }{V(u) e^{4 c t}} \lesssim_\eta
	e^{- 3 c t} \, .
\end{align}
By Borel-Cantelli, there exists $N(\omega,u) \geq 1$ with 
$\P\left(N(\cdot,u) \geq n\right)\lesssim_\eta e^{-3c n}$
such that 
\[
	\exp\left(\int_0^n\|\nabla u_s\|_{L^\infty}\ds\right) 
	\leq V(u)e^{4cn} \, \quad \text{for $n \geq N(\omega,u)$.}
\]
To bound when $n < N(\omega,u)$, we find
\[
	\exp\left(\int_0^n\|\nabla u_s\|_{L^\infty}\ds\right) \leq D_0(\omega,u) := \exp\left(\int_0^{N(\omega,u)}\|\nabla u_s\|_{L^\infty}\ds\right)
\]
and note that by Cauchy-Schwarz,
\begin{equation}
\begin{aligned}
	\E D_0(\cdot,u) &\leq \sum_{n}(\P(N = n))^{1/2} \left[ \E\exp\left(2 \int_0^{n}\|\nabla u_s\|_{L^\infty}\ds\right) \right]^{1/2}\\
	&\leq \sum_n  (e^{-3 c n})^{1/2} (V(u) e^{c n})^{1/2}  \leq V(u) \sum _n e^{-c n} \leqc V(u). \qedhere
\end{aligned}
\end{equation}
\end{proof}

With these preparations in place, we now state the main results of \cite{BBPS18, BBPS19I, BBPS19II}. 

\begin{theorem}[Theorem 1.2 in \cite{BBPS19II}] \label{thm:UniMix}
There exists $\kappa_0 > 0$ for which the following holds for all $\kappa \in [0,\kappa_0]$.
Let $(u_t)$ solve \eqref{eq:NS-Abstract} for an arbitrary initial condition $u_0 = u \in \Hbf$. 
For all $s > 0, p \geq 1$, there exists a deterministic $\gamma = \gamma(s, p) > 0$ (depending only on $s,p$ and the parameters $Q$, $\nu$ etc) and a random constant $D_{\kappa}(\omega,u): \Omega \times \Hbf \to [1,\infty)$ (also depending on $p, s$, as well as $\kappa$) such that for all $g\in H^s$,
\begin{align}
\norm{S^\kappa_t(\omega,u)g}_{H^{-s}} \leq D_\kappa(\omega,u) e^{-\gamma t} \norm{g}_{H^s} \, .
\end{align}
The random constant $D_\kappa$ has $V$-bounded $p$-th moment with implicit constant
independent of $\kappa$. 
\end{theorem}
Theorem \ref{thm:UniMix} is proved in \cite{BBPS19I} for $\kappa = 0$ and \cite{BBPS19II} for 
$\kappa \in [0,\kappa_0]$. Both papers rely heavily on Lagrangian chaos as proved in \cite{BBPS18}. From Theorem \ref{thm:UniMix}, it is relatively straightforward to prove the following
enhanced dissipation result. 
\begin{theorem}[Theorem 1.3 in \cite{BBPS19II}] \label{thm:ED}
Let $\kappa_0 > 0$ and $\gamma = \gamma(1,p)$ be as in Theorem \ref{thm:UniMix}, where
 $p \geq 2$ is arbitrary. Let $(u_t)$ solve \eqref{eq:NS-Abstract} for an arbitrary initial condition $u_0 = u \in \Hbf$. Then, there is a random constant $D_\kappa'(\omega,u)$ such that for all $g\in L^2$
\begin{align}
\norm{S^\kappa_t(\omega,u)g}_{L^2} \leq \min\{ 1 , D_{\kappa}'(\omega,u) \kappa^{-1} e^{- \gamma t} \} \norm{g}_{L^2}. \label{ineq:ED2}
\end{align}
The random constant $D_\kappa'$ has $V$-bounded $p$-th moment with implicit constant
independent of $\kappa$. 
\end{theorem}

\begin{theorem}[Theorem 1.5 in \cite{BBPS19II}]\label{thm:optTimeScale}
In the setting of Theorem \ref{thm:UniMix}, let
\[
\tau_* = \tau_\ast(\omega,u,g) = \inf\set{t: \norm{g_t}_{L^2} < \tfrac{1}{2}\|g\|_{L^2}} \, .
\]
Then, there exists a $\kappa_0 > 0$ a sufficiently small universal constant such that
for all $\kappa\in (0,\kappa_0]$, one has
\[
	\tau_\ast(\omega,u,g) \geq \delta(g,u,\omega) |\log{\kappa}| \quad\text{with probability}\quad 1 \, , 
\]
where $\delta (g, u, \omega) \in (0,1)$ is a $\kappa$-independent random constant
with the property that there exists a $\beta \geq 1$ such that for all $p \geq 1$ and $\eta>0$ with $V(u) = V_{\beta,\eta}(u)$ we have
\begin{align}
\E \delta^{-p}(\cdot,u) \leqc_{p,\eta, \beta} \frac{\|g\|_{H^1}^p}{\|g\|_{L^2}^p} V(u)^p \, . \label{ineq:deltap}
\end{align}
\end{theorem}

\subsection{Unique stationary measure for $(u_t, g_t^\kappa)$}\label{subsec:unqiueness}

We provide here an argument proving uniqueness of the stationary measure for the 
passive scalar process $(u_t, g_t^\kappa)$ on $\Hbf \times L^2$. The following argument is quite general, and applies in a variety of cases outside the scope of the main results; 
see Remark \ref{rmk:generalErgPSprocess} below. 

\begin{proposition}\label{prop:uniqueStatMeasPSP}
For any $\kappa > 0$, the Markov process $(u_t, g_t^\kappa)$ admits a unique
stationary measure $\mu^\kappa$ on $\Hbf \times L^2$. 
\end{proposition}

\newcommand{\Gc}{\mathcal{G}}

\begin{proof}
It suffices to show that any two ergodic stationary measures
$\mu_1^\kappa, \mu_2^\kappa$ for $(u_t, g_t^\kappa)$ coincide. 
By standard ergodic theory for random dynamical systems \cite{kifer2012ergodic}
and the ergodic decomposition theorem (see, e.g., \cite{einsiedler2013ergodic}),
uniqueness of ergodic stationary measures for $(u_t, g_t^\kappa)$ implies uniqueness of
stationary measures. 

Assume now that $\mu_i^\kappa$ are two ergodic stationary measures, $i = 1,2$. 
To prove $\mu_1^\kappa = \mu_2^\kappa$, it suffices to show 
that for each bounded, globally Lipschitz $\psi : \Hbf \times L^2 \to \R$, we have
that $\int \psi\, \dee \mu_1^\kappa = \int \psi\, \dee \mu_2^\kappa$. Without loss of generality, we can assume $\|\psi\|_{\mathrm{Lip}}=1$.

By ergodicity of $\mu^\kappa_i$ 
and the Birkhoff ergodic theorem, we have the following for $i = 1,2$: there 
is a set $\hat \Gc_i = \hat \Gc_i(\psi) \subset \Hbf \times L^2$ of full $\mu_i^\kappa$-measure
such that for all initial $(u,g) \in \hat \Gc_i$, we have
\begin{align}
\lim_{n \to \infty} \frac{1}{n} \sum_{t = 0}^{n-1} \psi(u_t, g_t^\kappa) = \int \psi\,\dee \mu_i^\kappa
\end{align}
with probability 1.
Next, observe that $\mu_i^\kappa$ projects to the unique stationary measure $\mu$ 
for the $(u_t)$ process (Proposition \ref{prop:uniqueStatMeasNSE}). It follows from Fubini's theorem that for $i = 1,2$, we have $\mu(\Gc_i) = 1$, where
$\Gc_i$ is the projection of $\hat \Gc_i$ to $\Hbf$, i.e., 
$\Gc_i = \{ u \in \Hbf : \exists g \in L^2 \text{ such that } (u, g ) \in \hat \Gc_i\}$. 
Since $\mu(\Gc_i) = 1, i = 1,2$, we conclude that $\Gc_1 \cap \Gc_2$ is nonempty. Fix
$u \in \Gc_1 \cap \Gc_2$ and let $g,h \in L^2$ be such that $(u, g) \in \hat \Gc_1$,
$(u, h) \in \hat \Gc_2$. Observe that 
\[
g_t - h_t = S^\kappa_t(\omega, u) (g - h) \, , 
\]
i.e., $g_t - h_t$ solves \eqref{def:AD}. 
Theorem \ref{thm:ED} implies $\| g_t - h_t\|_{L^2} \to 0$ 
as $t \to \infty$. Given $\delta > 0$, let $N_0 = N_0(\delta')$ be such that
$\| g_t - h_t\|_{L^2} < \delta$ for all $t \geq N_0$. 

With these preparations in place, fix $u \in \Gc_1 \cap \Gc_2$ and $g, h \in L^2$ such that
$(u,g) \in \hat \Gc_1, (u,h) \in \hat \Gc_2$. 
With $\delta > 0$ as above, let $N_0 = N_0(\delta)$. We have
\begin{align*}
\left| \int \psi\, \dee \mu_1^\kappa - \int \psi\, \dee \mu_2^\kappa\right| \leq
\lim_{n \to \infty} \frac{1}{n} \sum_{t  =0}^{n-1} |\psi(u_t, g_t^\kappa) - \psi(u_t, h_t^\kappa)| 
= \lim_{n \to \infty} \frac{1}{n} \sum_{t = N_0}^{n-1} |\psi(u_t, g_t^\kappa) - \psi(u_t, h_t^\kappa)| 
\end{align*}
Each summand on the RHS is $\leq \delta$, and so the whole limit is $\leq \delta$. 
Since $\delta > 0$ was arbitrary, this completes the proof. 
\end{proof}

\begin{remark}\label{rmk:generalErgPSprocess}
The proof given above has two main ingredients: (1) uniqueness for the stationary measure
$\mu$ on $\Hbf$, and (2) $L^2$ 
dissipation estimates for $S^\kappa_t(\omega, u)$ for $\kappa > 0$. Item (1) is known to 
hold for 2D Navier-Stokes under very weak nondegeneracy condition on the noise \cite{HM06}, while item (2) holds for a wide variety of fluid models by standard
estimates. Thus, the proof given above is applicable at a much higher level of generality
than that of Assumption \ref{ass:strong}, including the setting of 2D Navier-Stokes with
``truly hypoelliptic'' forcing as in \cite{HM06}. 
\end{remark}

\section{Proof of Batchelor's law}\label{sec:Batchelor-Proof}
% !TEX root = master.tex

In this section we prove the main result of this paper, 
Theorem \ref{thm:Batch}. To summarize the main approach, 
observe that by uniqueness of the stationary measure $\mu^\kappa$
for \eqref{eq:scalar-intro} (Proposition \ref{prop:uniqueStatMeasPSP}), the Birkhoff ergodic theorem implies
\[
\E_{\mu^\kappa} \| \Pi_{\leq N} g \|_{L^2}^{2} = \lim_{T \to \infty}
\frac{1}{T} \int_0^T \| \Pi_{\leq N} g_t^\kappa \|_{L^2}^{2} \dt
\]
for $\mu^\kappa$-generic initial $u = u_0  \in \Hbf, g = g_0^\kappa \in L^2$. 
On the other hand, the mild formulation for \eqref{eq:scalar-intro} reads as
\begin{equation}\label{eq:mild-form}
	g_t^\kappa = S^\kappa_{t}(\omega,u)g + \int_0^tS_{t,s}^\kappa(\omega,u)b\,\dee\beta_s \, .
\end{equation}
Using the fact that
\[
	\E_\beta\left\langle S_t^{\kappa}(\omega,u)g,\int_0^t\Pi_{\leq N}S_{t,s}^\kappa(\omega,u)b\,\dee \beta_s\right\rangle_{L^2} = 0
\]
results in the identity
\begin{align}\label{eq:p=1Case}
\E \| \Pi_{\leq N} g_t^\kappa\|_{L^2}^2 = \E \|  \Pi_{\leq N} S^\kappa_t(\omega, u) g \|_{L^2}^2 
+ \E \int_0^t \|  \Pi_{\leq N} S_{t, s}^\kappa(\omega, u) b\|_{L^2}^2 \ds
\end{align}
after taking an $L^2$ norm in $\T^d$, then an expectation, and finally the It\^o isometry. 
By Theorem \ref{thm:ED}, the first such term vanishes exponentially fast in $L^2$ as $t \to \infty$,
and so it is the second term that dominates the long-time behavior of 
$\E \| g_t^\kappa\|_{L^2}^2$, hence the value of 
$\E_{\mu^\kappa}\|\Pi_{\leq N}g\|_{L^2}^2$. 
While the above is stated for $p = 1$, the case $p > 1$ is handled similarly using 
the Burkholder-Davis-Gundy inequality.

The upper bound for \eqref{eq:PSIR} in the inertial range 
is carried out in Section \ref{subsec:upper}, and the lower bound in Section \ref{subsec:lowerBoundBatch}.
The dissipative-range upper bound is done in Section \ref{subsec:upperDissipative}. 

\subsection{Upper bound in the inertial range}\label{subsec:upper}
In this section we prove the $\lesssim$ direction of \eqref{eq:PSIR}, that is,
\begin{lemma} \label{lem:PSIRUpBd}
Under the conditions of Theorem \ref{thm:Batch}, there holds for all $p \in [1,\infty)$ and for all $2\leq N <\infty$,
\begin{align}
\left(\EE_{\mu^\kappa} \norm{\Pi_{\leq N} g}_{L^2}^{2p}\right)^{1/p} & \lesssim_p \log N.  \label{ineq:PSIRUpBd}
\end{align}
\end{lemma}

We begin with the following lemma.
\begin{lemma} \label{lem:StochUpBd}
For all $p\in [1,\infty)$ and $2 \leq N < \infty$,
\begin{align}
\limsup_{t \to \infty}  \int_0^t \left( \EE\norm{\Pi_{\leq N}S_{t,\tau}^\kappa(\omega,u)b}_{L^2}^{2p}\right)^{1/p} \dee\tau \lesssim_p \log N. 
\end{align}
\end{lemma}
\begin{proof}
Recall that $\norm{\Pi_{\leq N} h}_{L^2} \lesssim N\norm{ \Pi_{\leq N} h}_{H^{-1}}$ for $h \in L^2$. 
By \eqref{eq:SkewS} and Theorem \ref{thm:UniMix} followed by Lemma \ref{lem:RandCon}, 
\begin{align}
\int_0^t \left( \EE\norm{\Pi_{\leq N}S_{t,\tau}^\kappa(\omega,u)b}_{L^2}^{2p}\right)^{1/p} \dee\tau & 
= \int_0^t \left( \EE\norm{\Pi_{\leq N}S_{t-\tau}^\kappa(\theta_\tau\omega,u_\tau)b}_{L^2}^{2p}\right)^{1/p} \dee\tau \\
& \lesssim \int_0^t \left(\EE D^{2p}(\theta_\tau \omega,u_\tau)\right)^{1/p} \min(1, N^2 e^{-2\lambda(t-\tau)}) \dee\tau \\
& \lesssim \int_0^t \left(e^{-\lambda \tau}V(u) + K_\lambda \right) \min(1,N^2 e^{-2\lambda(t-\tau)}) \dee\tau \\
& \lesssim \left(e^{-\lambda t} V(u) + 1\right)\log N.
\end{align}
The proof then follows by sending $t\to \infty$. 
\end{proof}

We are now ready to complete the proof of Lemma \ref{lem:PSIRUpBd}. 
\begin{proof}[\textbf{Proof of Lemma \ref{lem:PSIRUpBd}}]
First observe that, for all $(u,g) \in \Hbf \times L^2$ there holds by \eqref{eq:mild-form}
\begin{align}
\norm{\Pi_{\leq N} g_t}_{L^2}^{2p} \lesssim_p \norm{\Pi_{\leq N} S_{t}^\kappa(\omega,u)g}_{L^2}^{2p}+ \norm{\int_0^t \Pi_{\leq N} S_{t,\tau}^\kappa(\omega,u) b\, \dee \beta_\tau}_{L^2}^{2p}.
\end{align}
From the initial data term, by Theorem \ref{thm:ED}, $\exists \lambda$ (depending on $p$) such that for a suitable $V$ as in \eqref{eq:V-def}, 
\begin{align}
\left(\EE \norm{\Pi_{\leq N} S_{t}^\kappa(\omega,u)g}_{L^2}^{2p} \right)^{1/p} \lesssim \kappa^{-1}V(u) e^{-\lambda t}\norm{g}_{L^2}^2 \, .
\end{align}
By Burkholder-Davis-Gundy [e.g. Theorem 5.2.4 \cite{DPZ96}] followed by Minkowski's inequality,
\begin{align}
\left(\EE\norm{\int_0^t \Pi_{\leq N}S_{t,\tau}^\kappa(\omega,u) b\, \dee \beta_\tau}_{L^2}^{2p}\right)^{1/p} & \lesssim_p \left(\EE\left(\int_0^t \norm{ \Pi_{\leq N} S_{t,\tau}^\kappa(\omega,u)}_{L^2}^2 \dee\tau\right)^p\right)^{1/p} \\
& \lesssim \int_0^t \left( \EE\norm{\Pi_{\leq N}S_{t,\tau}^\kappa(\omega,u)b}_{L^2}^{2p}\right)^{1/p} \dee\tau.
\end{align}
Hence, Lemma \ref{lem:StochUpBd} implies
\begin{align}
\limsup_{t \to \infty}\left(\EE\norm{\Pi_{\leq N}g_t}_{L^2}^{2p}\right)^{1/p} \lesssim \log N.
\end{align}
Note that by standard moment estimates (see e.g. \cite{KS}), for all $p,\kappa > 0$, $\EE_{\mu^\kappa} \norm{g}_{L^2}^p < \infty$.
Hence, by the Birkhoff ergodic theorem, for $\mu^\kappa$-a.e. $(u,g) \in \Hbf \times L^2$ we have
\begin{align}
\left(\EE_{\mu^\kappa} \norm{\Pi_{\leq N}g}_{L^2}^{2p}\right)^{1/p} = \lim_{T \to \infty}\frac{1}{T}\int_0^T\left(\EE_{(u,g)}\norm{\Pi_{\leq N}g_t}_{L^2}^{2p}\right)^{1/p}\dt  \lesssim \log N.
\end{align}
This completes the proof of Lemma \ref{lem:PSIRUpBd}.
\end{proof} 

\subsection{Lower bound in the inertial range}\label{subsec:lowerBoundBatch}
In this section we prove the $\gtrsim$ direction of \eqref{eq:PSIR}.

\begin{lemma} \label{lem:PSIRLowBd}
Under the conditions of Theorem \ref{thm:Batch}, there exists an $N_0 \geq 2$ chosen sufficiently large (depending on $b$ and the fluid parameters only) such that for all $\kappa \in (0,\kappa_0]$ the following holds for all $p \in [1,\infty)$ and all $N \in [N_0,\kappa^{-1/2}]$,
\begin{align}
\left(\EE_{\mu^\kappa} \norm{\Pi_{\leq N} g}_{L^2}^{2p}\right)^{1/p} & \gtrsim_{p} \log N.  \label{ineq:PSIRLowBd}
\end{align}
\end{lemma}

We begin with the following, which follows from an $H^1$ energy estimate and Lemma \ref{lem:twist-cantelli}.
\begin{lemma} \label{lem:H1grwth}
There holds (with $C_0$ and $D_0$ as in Lemma \ref{lem:twist-cantelli}) 
\begin{align}
\norm{S_{t}^\kappa(\omega,u)g}_{H^1}  \leq D_0(\omega,u) e^{C_0 t}\norm{g}_{H^1}.
\end{align}
\end{lemma}
    
The next lemma provides the lower bound on the stochastic integral contribution.  
\begin{lemma} \label{lem:LowBdMain}
There exists an $N_0 \geq 2$ chosen sufficiently large (depending on $b$ and the fluid parameters only) such that for all $\kappa \in (0,\kappa_0]$ the following limit holds for all $N \in [N_0,\kappa^{-1/2}]$,
\begin{align}
\liminf_{t \to \infty} \int_0^t \EE \norm{\Pi_{\leq N} S_{t-\tau}^\kappa(\omega,u)b}_{L^2}^2 \dee\tau \gtrsim \log N. 
\end{align}
\end{lemma}
\begin{proof}
Fix $u\in \Hbf$ and let $\gamma \in (0,1)$  to be chosen later. Write for $t \geq \gamma \log N$ and $N \geq 2$, 
\begin{align}
\int_0^t \|\Pi_{\leq N} S_{t-r}^\kappa(\theta_r\omega,u_r)b\|_{L^2}^2 \dr & \geq \int_{t-\gamma \log{N}}^t \|S_{t-r}^\kappa(\theta_r\omega,u_r)b\|_{L^2}^2 \dr \nonumber \\ & \quad - \int_{t-\gamma \log{N}}^t \|\Pi_{>N}S_{t-r}^\kappa(\theta_r\omega,u_r)b\|_{L^2}^2 \dr. \label{ineq:SimSplit}
\end{align}
The first term is controlled by the $H^1$ norm growth bound in Lemma \ref{lem:H1grwth} 
\begin{align}
\E \int_{t-\gamma\log N}^t \|\Pi_{> N} S_{t-r}^\kappa(\theta_r\omega,u_r)b\|_{L^2}^2 \dr \leqc \int_{t-\gamma \log N}^t N^{-2}\E D_0^2(\theta_r\omega,u_r)e^{\lambda (t-r)}\dr.
\end{align}
By Lemma \ref{lem:RandCon}
\begin{align}
\E D_0^2(\theta_r \omega, u_r)  \lesssim  e^{-\lambda r}V(u) + K_\lambda,
\end{align}
and hence, choosing $\gamma < \lambda^{-1}$, we have
\begin{align}
\limsup_{t\to \infty}\E \int_{t-\gamma \log N}^t \|\Pi_{> N} S_{t-r}^\kappa(\theta_r\omega,u_r)b\|_{L^2}^2 \dr% \\ &\hspace{1in}
& \leqc \gamma N^{\gamma\lambda -2} \log N \limsup_{t\to \infty}(e^{-\lambda t}V(u) + K_\lambda)\\
& \leq C_0 N^{-1}\gamma \log N, \label{ineq:SimSplit2}
\end{align}
for some constant $C_0$ universal independent of $\gamma, \kappa, t, u$. 
This completes the second term in \eqref{ineq:SimSplit}.
For the first term in \eqref{ineq:SimSplit}, we use the stopping time defined in Theorem \ref{thm:optTimeScale} with initial data $b$: 
\[
\E \|S_{t-r}^\kappa(\theta_r\omega,u_r)b\|_{L^2}^2 \geq \frac{1}{2}\|b\|_{L^2}^2 \P\left(\tau_\ast(\theta_{r}\omega,u_r) > t-r\right) \, , 
\]
where (analogously to Theorem \ref{thm:optTimeScale}) $\tau_\ast(\omega,u)$ is defined by
\[
	\tau_\ast(\omega,u) = \inf\left\{t>0 \,:\, \|S_{t}^\kappa(\omega,u)b\|_{L^2}^2\leq \frac{\norm{b}_{L^2}^2}{2}\right\}.
\]
Theorem \ref{thm:optTimeScale} then implies that $\exists$ $C_1 > 0$ (independent of $t,r,\kappa,u$) and $\delta(\omega,u) : \Omega \times \Hbf \to \R_{> 0}$ with $\E\delta^{-1}(\cdot,u) \leq C_1 V(u)$ such that $\tau_\ast \geq \delta |\log{\kappa}|$. 
Using Chebyshev's inequality and an argument similar to that used to prove Lemma \ref{lem:RandCon}
\begin{align}
\P\left(\tau_\ast(\theta_{r}\omega,u_r) > t-r\right)  & \geq  1 - \PP\left( \delta^{-1}(\theta_r \omega, u_r) > |\log\kappa|/(t-r)\right) \\
& \geq 1 - \frac{t-r}{\abs{\log \kappa}} \EE \left( \EE ( \delta^{-1}(\theta_r\omega,u_r) | \mathscr{F}_r)  \right) \\
& \geq 1 - C_1\frac{t-r}{|\log{\kappa}|}(e^{-\lambda r}V(u) + K_\lambda). 
\end{align}
It follows that for $2 \leq N \leq \kappa^{-1/2}$,
\[
\begin{aligned}
  & \liminf_{t\to \infty}\E \int_{t-\gamma \log N}^t \|S_{t-r}^\kappa(\theta_r\omega,u_r) b\|_{L^2}^2 \dr\\
&\hspace{1in}\geq \liminf_{t \to \infty} \frac{1}{2}\|b\|_{L^2}^2\gamma\log N  \left(1 - \frac{C_1 \gamma}{2} \frac{\log N}{|\log{\kappa}|}(e^{-\lambda t} V(u) + K_\lambda)\right)\\
&\hspace{1in}\geq \frac{1}{2}\|b\|_{L^2}^2\gamma \log N \left(1 - \frac{C_1 K_\lambda \gamma}{4} \right). 
\end{aligned}
\]
Therefore, for $K_\lambda C_1 \gamma < 1/8$ and $N > N_0 := 8C_0$.
\begin{align}
\liminf_{t\to \infty}\EE \int_0^t \|\Pi_{\leq N} S_{t-r}^\kappa(\theta_r\omega,u_r)b\|_{L^2}^2 \dr & \geqc (1- \frac{\gamma}{4} C_1K_{\lambda} - 2 C_0 N^{-1})\log N \\
& \geqc \log N . 
\end{align}
\end{proof}

\begin{proof}[\textbf{Proof of Lemma \ref{lem:PSIRLowBd}:}]
By Jensen's inequality, $\EE_{\mu^\kappa} \norm{\Pi_{\leq N} g}_{L^2}^{2p} \geq ( \EE_{\mu^\kappa} \norm{\Pi_{\leq N} g}^2_{L^2} )^p$ for $p \geq 1$, and so in the following argument it suffices to consider the $p = 1$ case. 
For arbitrary initial data  $g \in L^2$, we have by \eqref{eq:p=1Case} that
\begin{align}
\E \norm{ \Pi_{\leq N} g_t}_{L^2}^2 \geq \EE \int_0^t \|\Pi_{\leq N} S_{t-r}^\kappa(\theta_r\omega,u_r)b\|_{L^2}^2 \dr \, , 
\end{align}
hence $\liminf_{t \to \infty} \E \norm{ \Pi_{\leq N} g_t}_{L^2}^2 \gtrsim \log N$ by Lemma 
\ref{lem:LowBdMain} for $N_0  < N < \kappa^{-1/2}$.
By the Birkhoff ergodic theorem (as in Lemma \ref{lem:PSIRUpBd} above), 
\begin{align}
\EE_{\mu^\kappa}\norm{\Pi_{\leq N} g_t}_{L^2}^{2} & =  \lim_{T \to \infty}\frac{1}{T}\int_0^T \EE \norm{\Pi_{\leq N} g_t}_{L^2}^{2} \dt \gtrsim \log N \, . 
\end{align}
This completes the proof of Lemma \ref{lem:PSIRLowBd}.
\end{proof}

\subsection{Upper bound in the dissipative range}\label{subsec:upperDissipative}

First, we derive the upper bound on the $L^2$ norm stated in \eqref{ineq:preg}. 
If we were not also interested in higher moments we could use the following simple observation using the $L^2$ balance \eqref{eq:L2Balance} and the upper bound in Batchelor's law \eqref{ineq:PSIRUpBd}
\begin{align}
\EE_{\mu^\kappa}\norm{g}_{L^2}^2 & \lesssim \EE_{\mu^\kappa}\norm{\Pi_{\leq \kappa^{-1/2}}g}_{L^2}^2 + \kappa \EE_{\mu^\kappa} \norm{\Pi_{> \kappa^{-1/2}}\grad g}_{L^2}^2 \lesssim 1 + \abs{\log \kappa}. 
\end{align}
The following argument provides also estimates on moments $p > 2$.
\begin{lemma} \label{lem:L2upbd}
Under the conditions of Theorem \ref{thm:Batch}, there holds for all $p \in [1,\infty)$, for all $\kappa \in (0,\kappa_0]$
\begin{align}
\left(\EE_{\mu^\kappa} \norm{\Pi_{\leq N} g}_{L^2}^{2p}\right)^{1/p} & \lesssim_p \abs{\log \kappa}.  \label{ineq:L2upBd}
\end{align}
\end{lemma}
\begin{proof}
  The proof proceeds similar to Lemma \ref{lem:PSIRLowBd} with the enhanced dissipation estimate (Theorem \ref{thm:ED}) in place of the uniform mixing estimate (Theorem \ref{thm:UniMix}). As in Lemma \ref{lem:PSIRLowBd}, by Burkholder-Davis-Gundy and Minkowski's inequality, 
\begin{align}
\left(\EE \norm{g_t}_{L^2}^{2p}\right)^{1/p} \lesssim \left(\EE \norm{S_t^\kappa(\omega,u) g}_{L^2}^{2p}\right)^{1/p} + \int_0^t \left( \EE \norm{S_{t,\tau}^\kappa(\omega,u) b}_{L^2}^{2p} \right)^{1/p} \dee \tau.  
\end{align}
By Theorem \ref{thm:ED}, $\exists \lambda >0$ and $D^\prime_\kappa(\omega,u)$ with $V$-bounded $2p$-th moments we have for any $(u,g) \in \Hbf \times L^2$,
\begin{align}
\left(\EE \norm{g_t}_{L^2}^{2p}\right)^{1/p} & \lesssim \left(\EE \norm{S_t^\kappa(\omega,u) g}_{L^2}^{2p}\right)^{1/p} + \int_0^t \left( \EE \norm{S_{t,\tau}^\kappa(\omega,u) b}_{L^2}^{2p} \right)^{1/p} \dee \tau \\ 
& \lesssim \left(\EE (D_\kappa'(\omega,u))^{2p}\right)^{1/p} \kappa^{-1} e^{-\lambda t} \norm{g}_{L^2}^{2} \\
& \quad + \int_0^t \left(\EE (D_\kappa'(\theta_{\tau} \omega,u_\tau))^{2p}\right)^{1/p} \min(1, \kappa^{-1}e^{-\lambda(t-\tau)}) \dee\tau. 
\end{align}
By  Lemma \ref{lem:RandCon}, we therefore have 
\begin{align}
\left(\EE \norm{g_t}_{L^2}^{2p}\right)^{1/p} & \lesssim_p  V(u) \kappa^{-1} e^{-\lambda t} \norm{g}_{L^2}^2 + \int_0^t (e^{-\lambda \tau}V(u) + 1)\min(1, \kappa^{-1}e^{-\lambda(t-\tau)}) \dee\tau \\
& \lesssim_p V(u) \kappa^{-1} e^{-\lambda t}(1 + \norm{g}_{L^2}^2) + \abs{\log \kappa}. 
\end{align}
Hence for all $(u,g) \in \Hbf \times L^2$,
\begin{align}
\limsup_{t \to \infty} \EE\norm{g_t}^{2p}_{L^2} \lesssim \abs{\log \kappa}.
\end{align}
Finally, by the Birkhoff ergodic theorem, for $\mu^\kappa$ a.e. $(u,g) \in \Hbf \times L^2$
\begin{align}
\EE_{\mu^\kappa} \norm{g}_{L^2}^{2p} = \lim_{T\to \infty} \frac{1}{T}\int_0^T \EE_{(u,g)}\norm{g_t}_{L^2}^{2p} \dt \lesssim \abs{\log\kappa}^p. 
\end{align}
\end{proof}

Finally we turn to the proof of \eqref{ineq:preg} for $q > 0$, which is a relatively straightforward consequence of parabolic regularity and the $L^2$ a priori estimate in Lemma \ref{lem:L2upbd}. 
First we prove the following quantitative $H^s$ regularization estimate.
Below we use the Fourier multiplier notation: for $m:\C^d \to \R$ measurable, we define the operator
\begin{align}
m(\grad)f = \left( m(i k) \hat{f}(k) \right)^{\vee}. 
\end{align}
\begin{lemma} \label{lem:parabolic}
For all initial data $(u,g) \in \Hbf \times L^2$, for all $ \gamma \in (0, \alpha - d - 1)$, $\exists \beta \geq 0$, such that $\forall p \in [1,\infty)$ and $\forall \eta>0$ there holds for $V = V_{\beta,\eta}$
\begin{align}
\EE\norm{\left(1 + \sqrt{\kappa} \brak{\grad}\right)^\gamma S^\kappa_{1}(\omega,u) g}_{L^2}^p \lesssim_{p,\gamma} V^p(u) \norm{g}_{L^2}^p. 
\end{align}
\end{lemma}
\begin{proof}
We first deduce a pathwise a priori estimate assuming $g$ is sufficiently smooth. 
To this end, by Plancherel's identity we have for $t \in [0,1]$, 
\begin{align*}
  \frac{1}{2} \frac{\dee}{\dt} \norm{\left(1  + \sqrt{\kappa} \brak{\grad}\right)^{\gamma t}   g_t^\kappa   }_{L^2}^2  & = \sum_{k \in \mathbb Z_\ast^d} \left(\gamma \log \left(1 + \sqrt{\kappa}\brak{k}\right) - \kappa \abs{k}^2\right) \left(1 + \sqrt{\kappa} \brak{k} \right)^{2\gamma t}\abs{\hat{g}^\kappa_t(k)}^2  \\
& \quad+ \brak{\left(1  + \sqrt{\kappa} \brak{\grad}\right)^{\gamma t} g_t^\kappa, \left(1  + \sqrt{\kappa} \brak{\grad}\right)^{\gamma t}\grad \cdot (u_t g_t^\kappa)}_{L^2} \\
& =: \mathcal{I}_1 + \mathcal{I}_2.
\end{align*}
For $\mathcal{I}_1$, by $\gamma \log (1 + x) \leq x^2 + C(\gamma)$,
\begin{align}
\mathcal{I}_1 \lesssim_\gamma  \norm{\left(1  + \sqrt{\kappa} \brak{\grad}\right)^{\gamma t} g_t^\kappa}^2. 
\end{align}
By incompressibility and Plancherel's identity
\begin{align}
\mathcal{I}_2 & = \brak{\left(1  + \sqrt{\kappa} \brak{\grad}\right)^{\gamma t} g_t^\kappa, \left(1  + \sqrt{\kappa} \brak{\grad}\right)^{\gamma t}\grad \cdot (u_t g_t^\kappa) - \grad \cdot (u_t \left(1  + \sqrt{\kappa} \brak{\grad}\right)^{\gamma t}g_t^\kappa) }_{L^2} \\
& = \sum_{k,\ell \in \mathbb Z^d_\ast}  \left(1 + \sqrt{\kappa} \brak{k}\right)^{\gamma t} \overline{\hat{g}_t^\kappa}(k) \left(\left(1 + \sqrt{\kappa} \brak{k}\right)^{\gamma t} - \left(1 + \sqrt{\kappa} \brak{k-\ell}\right)^{\gamma t} \right) ik \cdot \hat{u}_t(\ell) \hat{g}^\kappa_t(k-\ell) \\
& = \sum_{k,\ell \in \mathbb Z^d_\ast} \left(\mathbf{1}_{\abs{k-\ell} \geq \frac{1}{2}\abs{\ell}} + \mathbf{1}_{\abs{k-\ell} < \frac{1}{2}\abs{\ell}} \right) \left(1 + \sqrt{\kappa} \brak{k}\right)^{\gamma t} \overline{\hat{g}^\kappa_t}(k) \\ & \quad\quad \times \left(\left(1 + \sqrt{\kappa} \brak{k}\right)^{\gamma t} - \left(1 + \sqrt{\kappa} \brak{\ell}\right)^{\gamma t} \right) ik \cdot \hat{u}_t(k-\ell) \hat{g}^\kappa_t(\ell) \\
& = \mathcal{I}_{2;HL} + \mathcal{I}_{2;LH}. 
\end{align}
On the support of the $\mathcal{I}_{2;HL}$ (here $HL$ stands for ``high-low''), we use that $\abs{k} + \abs{\ell} \lesssim \abs{k-\ell}$, together with Cauchy-Schwarz,  Young's inequality, $H^r \hookrightarrow \widehat{L^1}$ for $r > \frac{d}{2}$, to deduce that for any $r \in (\frac{d}{2}+1,3)$,
\begin{align}
\abs{\mathcal{I}_{2;HL}} \lesssim \norm{\left(1 + \sqrt{\kappa} \brak{\grad}\right)^{\gamma t} g_t^\kappa}_{L^2} \norm{g_t^\kappa}_{L^2} \norm{\left(1 + \sqrt{\kappa} \brak{\grad}\right)^{\gamma t} u_t}_{H^{r}}. 
\end{align}
On the support of the $\mathcal{I}_{2;LH}$, we use the mean value theorem and $\abs{k} \lesssim \abs{\ell}$ to deduce 
\begin{align}
\left(1 + \sqrt{\kappa} \brak{k}\right)^{\gamma t} - \left(1 + \sqrt{\kappa} \brak{k-\ell}\right)^{\gamma t}  \lesssim  \sqrt{\kappa} \gamma t \left(1 + \sqrt{\kappa} \brak{k-\ell}\right)^{\gamma t - 1} \abs{\ell}.
\end{align}
Therefore, by Cauchy-Schwarz, Young's inequality, and $H^r \hookrightarrow \widehat{L^1}$ for $r > \frac{d}{2}$, for any $r \in (\frac{d}{2}+1,3)$ we have
\begin{align}
\abs{\mathcal{I}_{2;LH}} \lesssim  \norm{u_t}_{H^r} \norm{\left(1 + \sqrt{\kappa} \brak{\grad}\right)^{\gamma t} g_t^\kappa}_{L^2}^2. 
\end{align}
By Gr\"onwall's inequality, there is some $C > 0$ (independent of $g$, $\kappa$, $u$, $t$) such that the following holds, 
\begin{align*}
\norm{\left(1  + \sqrt{\kappa} \brak{\grad}\right)^{\gamma t} g_t^\kappa   }_{L^2}^2 & \leq \exp\left(Ct + \int_0^t \norm{u_\tau}_{H^{r}} \dee\tau\right) \norm{g}_{L^2}^2 \\
 & \quad + \exp\left(Ct + \int_0^t \norm{u_\tau}_{H^{r}} \dee\tau\right) \int_0^t \norm{\left(1  + \sqrt{\kappa} \brak{\grad}\right)^{\gamma \tau}u_\tau}_{H^{r}}^2 \norm{g_\tau}^2_{L^2} \dee \tau. 
\end{align*}
Taking expectations and applying Lemma \ref{lem:TwistBd} completes the a priori estimate \eqref{lem:parabolic}. 
For an arbitrary $g \in L^2$, the desired result follows by density. 
\end{proof}

\begin{lemma}
Under the assumptions of Theorem \ref{thm:Batch}, \eqref{ineq:preg} holds. 
\end{lemma} 
\begin{proof}
By standard moment estimates (see e.g. \cite{KS}), $\EE_{\mu^\kappa} \norm{g}_{H^s}^{2p} < \infty$ for $s \in [0,\sigma - \tfrac{3d}{2} -1 )$ and $p < \infty$.   
By stationarity,
\begin{align}
  \E_{\mu^\kappa}\|g\|_{H^s}^{2p} &= \int \EE_{(u,g)}\|g_1^\kappa\|_{H^s}^{2p}\, \dee\mu^\kappa(u,g) \\
  & \lesssim_p \int \EE \|S_1^\kappa(\omega,u) g\|_{H^s}^{2p}\, \dee\mu^\kappa(u,g) + \int\EE \norm{\int_0^1 S_{1-\tau}^\kappa(\theta_\tau\omega,u_\tau) b \, \dee\beta_\tau }_{H^s}^{2p}  \dee\mu^\kappa(u,g). 
\end{align}
By Lemma \ref{lem:parabolic}, $\exists \beta \geq 0$ such that for $\forall \eta>0$, $V = V_{\beta,\eta}$, and all $(u,g) \in \Hbf \times L^2$, 
\begin{align}
\E \norm{S_1^\kappa(\omega,u)g}_{H^s}^{2p} \lesssim \kappa^{-sp}V^{p}(u) \norm{g}_{L^2}^{2p}. 
\end{align}
Because $b\in C^\infty$, it follows from Burkholder-Davis-Gundy, Minkowski's inequality, and classical $H^s$ estimates for the transport equation (together with Lemma \ref{lem:RandCon}) that $\exists \beta \geq 0$ such that $\forall \eta > 0$, there holds for suitable $V =V_{\beta,\eta}$,  
\begin{align}
\EE \norm{\int_0^1 S_{1-\tau}^\kappa(\theta_\tau\omega,u_\tau) b \, \dee\beta_\tau }_{H^s}^{2p}
\lesssim \int_0^1 \E \norm{S_{1 - \tau}^\kappa(\theta_\tau \omega, u_\tau) b}_{H^s}^{2p} \dee\tau
\lesssim V^p(u). 
\end{align}
Therefore by Lemmas \ref{lem:L2upbd}, \ref{lem:parabolic}, and H\"older's inequality, 
\begin{align}
\E_{\mu^\kappa}\|g\|_{H^s}^{2p} & \lesssim \E_{\mu^\kappa}\left[ V^{p}(u) \left(1 + \kappa^{-sp}\norm{g}_{L^2}^{2p}\right)\right] \lesssim_{p} \kappa^{-sp}\abs{\log \kappa}^{p}.
\end{align}
This completes the proof of \eqref{ineq:preg}. 
\end{proof}

\section{Vanishing diffusivity limit} \label{sec:Inviscid-Limit}
%!TEX root = master.tex

This section is devoted to the proof of Theorem \ref{thm:kappato0}
describing weak limits of the stationary measures $\mu^\kappa$ 
for the passive scalar process $(u_t, g_t^\kappa)$. 
In Section \ref{subsec:tightness} we show that such weak limits exist in the weak topology of of measures on $\Hbf\times H^{-s}$ and satisfy Batchelor's Law
\eqref{eq:infinite-Batchelor} over an infinite inertial range.

Once these are established, it remains to show that $\mu^0$ is a stationary
measure for the passive scalar process $(u_t, g_t^0)$. As we check in 
Section \ref{subsec:restrictedFeller}, the latter can be extended to a Markov 
process defined by a random dynamical system on $\Hbf\times H^{-1}$, the trajectories of which are weak solutions to the passive scalar advection equation at $\kappa = 0$. However, to check that weak limits $\mu^0$ are stationary is a non-trivial task due to poor continuity properties of $(u_t, g_t^0)$ in $H^{-1}$ with respect perturbations of the initial data $u\in\Hbf$ for the velocity field $(u_t)$, 
and so stationarity of $\mu^0$ requires careful justification-- see Section \ref{sec:restrictedFeller} for more discussion.
The proof is completed in Section \ref{subsec:concludeStationarity}.

Finally we complete this section with a proof of the $L^2$ non-vanishing flux relation \eqref{eq:AD212}.

\subsection{Existence of weak limits of $\{\mu^\kappa\}$}\label{subsec:tightness}

Corollary \ref{cor:unifHdeltaEstBatch} of Batchelor's law in Theorem \ref{thm:kappato0} 
is that for each $s \in (0,1]$ we have the uniform in $\kappa$ estimate
\begin{align}\label{eq:unifKappaEst}
	\sup_{\kappa > 0}\E_{\mu^\kappa}\|g\|_{H^{-s}}^{2p} \lesssim_{s, p} 1 \, .
\end{align}
Since $H^{-s/2}$ is compactly embedded in $H^{-s}$ for all $s > 0$, we can use this to 
show tightness of $\{ \mu^\kappa\}$, which permits us to take weak limits by Prokhorov's theorem.
Precisely, we have the following.
\begin{lemma}
There exists a sequence $\{\kappa_n\}$, $\kappa_n \to 0$ as $n\to \infty$ and a measure $\mu^0$ on $\Hbf\times L^2$ such that for each $s \in (0,1]$, the following holds:
$\mu^{\kappa} \to \mu^0$ weakly as a measure on $\Hbf\times H^{-s}$, i.e., for each $\phi \in C_b(\Hbf\times H^{-s})$ we have
\[
	\int \phi\, \dee\mu^{\kappa_n} \to \int \phi \,\dee \mu^0 \, .
\]
\end{lemma}

\begin{proof}
Fix $s, \ep > 0$ and define
\[
	K = \{h \in H^{-s}\,:\, \|h\|_{H^{-s/2}}^2\leq \ep^{-2}\E_{\mu^\kappa}\|g\|_{H^{-s/2}}^2\} \, .
\]
In light of the compact embedding $H^{-s/2}\hookrightarrow H^{-s}$, $K$ is a compact subset of $H^{-s}$; by \eqref{eq:unifKappaEst} and Chebyshev's inequality, we have
\[
	\sup_{\kappa\in (0,\kappa_0]} \mu^\kappa(\Hbf\times(H^{-s}\backslash K)) \leq \ep \, .
\]
It follows by Prokhorov's theorem that $(\mu^\kappa)_{\kappa > 0}$ has a subsequence that converges weakly as a measure on $\Hbf\times H^{-s}$. Choosing $s = s_j = 2^{-j}$ we can extract a diagonal subsequence $\kappa_n$ and a limit measure $\mu^0$ such that $\mu^{\kappa_n}\to\mu^{0}$ weakly as a measure on $\Hbf\times H^{-s_j}$ for every $j > 0$, hence on $\Hbf \times H^{-s}$ for every $s \in (0,1]$. 
\end{proof}

It is now straightforward to verify the infinite-inertial range version of Batchelor's law as in 
\eqref{eq:infinite-Batchelor}. 
With $\mu^0 = \lim_n \mu^{\kappa_n}$ 
fixed as above, observe that $g \mapsto \| \Pi_{\leq N} g \|_{L^2}$
varies continuously in $H^{-s}$ and by a straight forward truncation argument along with the fact that
\[
	\sup_{\kappa\in [0,1]}\E_{\mu^\kappa}\|\Pi_{\leq N}g\|_{L^2}^{2p} < \infty
\]
to conclude that
\[
	\E_{\mu^0}\| \Pi_{\leq N} g \|_{L^2}^{2p} = \lim_{n\to\infty} \E_{\mu^{\kappa_n}}\|\Pi_{\leq N}g\|_{L^2}^{2p} \approx_p (\log{N})^p,
\]
for all $N \geq N_0$, using the fact that the inertial range $N_0 \leq N \lesssim \kappa^{-1/2}$ 
grows to all of $\{ N \geq N_0\}$ as $\kappa \to 0$. 

\subsection{The scalar process $(u_t, g_t^0)$ at $\kappa = 0$}\label{subsec:restrictedFeller}

In this section we study the properties of the passive scalar process $(u_t, g_t^0)$ which solves the advection equation
\begin{equation}\label{eq:kappa=0-eq}
	\partial_t g_t^0 + u_t\cdot\nabla g_t^0 = b\dot{\beta}_t,
\end{equation}
in the absence of any diffusion term. Since the limiting measure $\mu^0$ assigns full measure to $H^{-1}$, it is natural to consider $t \mapsto g_t^0$ as a process in the negative Sobolev space $H^{-1}$, for which
it will be necessary to show that weak solutions to \eqref{eq:kappa=0-eq} are well-posed for $H^{-1}$ initial data. 

To formulate this, recall that the solution operator $S^0_t(\omega, u) : L^2 \to L^2$, defined by
\[
S_t^0(\omega, u) g = g \circ (\phi^t_{\omega, u})^{-1}
\]
solves \eqref{eq:kappa=0-eq} in the absence of the noise term $b \dot{\beta}_t$, where 
$\phi^t_{\omega, u} : \T^d \to \T^d$ is the Lagrangian flow associated to $(u_t)$. By incompressibility, for 
$f \in H^1, g \in L^2$ we have
\begin{align}\label{eq:dualSalbatross}
\int f S_t^0(\omega, u) g \,\dx = \int (f \circ \phi^t_{\omega, u}) g\, \dx \, .
\end{align}
\begin{lemma}\label{lem:H-1-semigroup}
The following holds
\begin{itemize}
\item[(a)] For each $t>0$ operator $S_t^0(\omega, u)$ defined by \eqref{eq:kappa=0-eq} admits a unique extension to a bounded linear operator
on $H^{-1}$ satisfying \eqref{eq:dualSalbatross} for all $g \in L^2, f \in H^{-1}$ and for all $g\in H^{-1}$, $t\mapsto S_{t}(\omega,u)g$ is strongly continuous in $H^{-1}$.
\item[(b)] For initial data $g \in H^{-1}$ the unique weak solution $g_t^0$ to equation \eqref{eq:kappa=0-eq} is is given by
\[
g_t^0 = S_{t}^0(\omega,u)g + \int_0^t S_{t-s}^0(\theta_s\omega,u_s)b \,\dee \beta_s \, .
\]
\end{itemize}
\end{lemma}

\begin{proof}
Part (b) follows from (a), which in turn follows from a density argument
and the following estimate for $g\in L^2$
\begin{align}\label{eq:H-1bound}
\| S^0_t(\omega, u) g \|_{H^{-1}} \leq \bar D(\omega, u) e^{c t} \| g \|_{H^{-1}},
\end{align}
where $c > 0$ and for each $p \in [1,\infty)$ we can take $\bar D(\omega, u) : \Omega \times \Hbf \to \R_{\geq 1}$ to have $V$-bounded $p$-th moment. To see this, fix $f \in H^1, g \in L^2$, and observe that by Gr\"{o}nwall's inequality and Lemma \ref{lem:twist-cantelli}
\begin{equation}
\begin{aligned}
	|\nabla (f\circ \phi_t^{\omega,u})(x)| &\leq \exp\left(\int_0^t\|\nabla u_s\|_{L^\infty} \ds\right)|\nabla f(\phi^t_{\omega, u}(x))|\\
	&\leq D_0(\omega,u)e^{ct}|\nabla f(\phi^t_{\omega, u}(x))|,
\end{aligned}
\end{equation}
 where
$c > 0$ and $D_0(\omega, u)$ has $V$-bounded $p$-th moment. Using incompressibility
of $\phi^t_{\omega, u}$, we obtain
\[
	\| f \circ \phi^t_{\omega, u}\|_{H^1} 
\leq D_0(\omega, u) e^{c t} \| f \|_{H^1},
\]
hence
\[
\left|\int S_t^0(\omega, u) g f \dx \right| = \left| \int g ( f \circ \phi^t_{\omega, u}) \dx \right| 
\leq \| g \|_{H^{-1}} \| f \circ \phi^t_{\omega, u} \|_{H^1} 
\leq  D_0(\omega, u) e^{c t} \| g \|_{H^{-1}} \|f\|_{H^1} \, .
\]
We conclude $\| S_t^0(\omega, u) g \|_{H^{-1}} \leq D_0(\omega, u) e^{c t} \| g \|_{H^{-1}}$, as
desired. Strong continuity of $t\mapsto S_t^0$ follows by density of $C^\infty(\T^d)$ in $H^{-1}$.
\end{proof}

\subsection{Stationarity of $\mu^0$: restricted Feller property of $(u_t, g_t^0)$}\label{sec:restrictedFeller}

With the evolution $t \mapsto g_t^0$ defined, let $P_t^0$ denote its corresponding Markov semigroup: 
\[
P_t^0 \phi(u, g) := \E_{(u,g)} \phi(u_t, g_t^0)
\]
for each bounded measurable $\phi : \Hbf \times H^{-1} \to \R$. We seek to show that if $\mu^0$ is a weak limit of $\mu^{\kappa_n}, \kappa_n \to 0$, then
$\mu^0$ is stationary for $(u_t, g_t^0)$; equivalently it suffices to show that 
for all globally Lipschitz, bounded $\phi : \Hbf \times H^{-1} \to \R$, we have
\[
\int P_1^0 \phi\, \dee \mu^0 = \int \phi\, \dee \mu^0 \, .
\]
By the weak convergence $\mu^{\kappa_n} \to \mu^0$ and 
stationarity of $\mu^{\kappa_n}$, it suffices to show that
\[
\lim_{n\to\infty}\int P_1^{\kappa_n} \phi \,\dee \mu^{\kappa_n} = \int P_1^0 \phi \,\dee \mu^0,
\]
 where $P^\kappa_t$ is 
the Markov semigroup on $\Hbf \times L^2$ for the 
process $(u_t, g_t^\kappa)$. To this end the strategy is to write
\begin{equation}\label{eq:two-terms}
	\left|\int P_1^{\kappa_n}\phi \,\dee \mu^{\kappa_n} - \int P_1^0 \phi\, \dee \mu^0\right| \leq \left|\int (P_1^{\kappa_n}\phi - P_1^0\phi) \,\dee \mu^{\kappa_n}\right| +  \left|\int P_1^0 \phi\, (\dee \mu^{\kappa_n} - \dee \mu^0)\right| \,
\end{equation}
and show that each term vanishes as $n\to \infty$. The first term is relatively straightforward to bound, having to do with the nearness of 
$g_t^{\kappa_n}, g_t^0$ when initiated at the same initial condition $g \in L^2$
and subjected to the same noise sample $b \beta_t$. 
This is dealt with by Lemma \ref{lem:firstTermZeroSparrow} below. 

The second term in \eqref{eq:two-terms} is more challenging. 
One would like to use weak convergence to justify 
passing the limit using weak convergence $\mu^{\kappa_n} \to \mu^0$, but this
will not work: to the best of our knowledge, the semigroup $P_t^0$ for $(u_t, g_t^0)$
does not send continuous functions on $\Hbf\times H^{-1}$ to continuous functions $\Hbf\times H^{-1}$, namely $P_t^0$ is not {\em Feller} on $\Hbf\times H^{-1}$. When $\kappa > 0$ it is a straightforward consequence of parabolic regularity that the 
semigroup $P_t^\kappa$ is Feller on $\Hbf\times L^2$. However, when $\kappa =0$ this mechanism is not available and consequently $P_t^0$ is not necessarily Feller on $\Hbf\times H^{-1}$ due to a lack of stability of the transport equation in $H^{-1}$ under perturbations of the velocity field. 

Instead, we will show the following \emph{restricted Feller property}: 
namely $P_t^0$ has the Feller property in $\Hbf\times H^{-1}$ when restricted to initial $g$ in a bounded subset of $L^2$. We state this property below for Lipschitz observables.

\begin{lemma}[Restricted Feller]\label{lem:restricted-feller}
Let $\phi \in \mathrm{Lip}(\Hbf\times H^{-1})$, then for all $u,u^\prime\in \Hbf$ and $g,g^\prime \in L^2$ we have
\[
	|P_1^0\phi(u,g) - P_1^0\phi(u^\prime,g^\prime)| \leqc \|\phi\|_{\mathrm{Lip}} W(u,g,u^\prime,g^\prime)\left( \|u- u^\prime\|_{\Hbf} + \|g - g^\prime\|_{H^{-1}}\right),
\]
where
\[
	W(u,g,u^\prime,g^\prime) = (V(u) + V(u^\prime))(1+\|g\|_{L^2})(1+\|g^\prime\|_{L^2}),
\]
and $V = V_{\beta,\eta}$ for a sufficiently large universal $\beta > 0$ and all $\eta>0$.
\end{lemma}

\begin{proof}
To prove this it suffices to consider two cases: (i) $u= u^\prime$ and $g\neq g^\prime$ and (ii) $u \neq u^\prime$, $g^\prime = g$. 

Case (i) is straightforward since the difference between 
any two solutions $g_t$ and $g_t^\prime$ to \eqref{eq:kappa=0-eq} 
with different initial data and the same noise and same velocity field $(u_t)$
 immediately satisfy
\[
	g_t - g_t^\prime = S_t^0(\omega,u)(g-g^\prime) \, , 
\]
hence by \eqref{eq:H-1bound},
\begin{equation}
\begin{aligned}
	|P_1^0\phi(u,g) - P_1^0(u,g^\prime)| &\leq \|\phi\|_{\mathrm{Lip}} \E \bar D(u,\cdot) \|g - g^\prime\|_{H^{-1}}\\
	&\leqc \|\phi\|_{\mathrm{Lip}} V(u) \|g - g^\prime\|_{H^{-1}}.
\end{aligned}
\end{equation}

For case (ii) let $u_t$ and $u_t^\prime$ be two solutions to \eqref{sys:NSE} with initial data $u$ and $u^\prime$ respectively and $g_t$ and $g^\prime_t$ the solutions to \eqref{eq:kappa=0-eq} with velocity fields $(u_t)$ and $(u_t^\prime)$ respectively and the same initial data $g\in H^1$. Then the difference $\tilde{g}_t = g_t - g^\prime_t$ satisfies
\[
	\partial_t\tilde{g}_t + u_t\cdot\nabla \tilde{g}_t + (u_t-u_t^\prime)\cdot\nabla g^\prime_t = 0, \quad \tilde{g}_0 = 0,
\]
and therefore can we written as
\[
	\tilde{g_t} = \int_0^t S_{t-s}^0(\theta_s\omega,u_s)\left[(u_s - u^\prime_s)\cdot \nabla g_s^\prime \right]\ds.
\]
It follows from \eqref{eq:H-1bound} (choosing $\bar D$ to have a $V$-bounded 4th moment) that for $0\leq s<t\leq 1$
\begin{equation}
\begin{aligned}
	\|S_{t-s}^0(\theta_s\omega,u_s)\left[(u_s - u^\prime_s)\cdot \nabla g_s^\prime \right]\|_{H^{-1}} &\leq \bar D(\theta_s\omega,u_s)\|(u_s - u^\prime_s)\cdot \nabla g_s^\prime\|_{H^{-1}}\\
	&\leq \bar D(\theta_s\omega,u_s)\|u_s - u^\prime_s\|_{W^{1,\infty}}\|g_s^\prime\|_{L^2},
\end{aligned}
\end{equation}
where in the last inequality we used that if $f\in W^{1,\infty}$ and $h\in H^{-1}$ then $\|fh\|_{H^{-1}}\leq \|f\|_{W^{1,\infty}}\|h\|_{H^{-1}}$, which can be seen by duality between $H^1$ and $H^{-1}$. Consequently we have
\begin{equation}
\begin{aligned}
	\E\|\tilde{g}_1\|_{H^{-1}}
	&\leq \int_0^1 \E\left[ \bar D(\theta_{s}\omega,u_s) \|u_s -u_s^\prime\|_{\Hbf}\|g^\prime_s\|_{L^2}\right]\ds
\end{aligned}
\end{equation}
We use H\"older's inequality under the expectation $\EE$. The 4th moment $\E \bar D(\omega, u)^4$
is bounded $\lesssim V(u)$ by construction, while by 
Lemma \ref{lem:uStablEst} we have 
$\E\|u_s - u^\prime_s\|_{\Hbf}^4 \lesssim (V(u) + V(u')) \| u - u'\|_{\Hbf}$. 
To bound $\E \| g^\prime_s\|_{L^2}^2$, we estimate as follows:
by It\^o's formula,
\[
\| g_s'\|_{L^2}^2 = \| g\|_{L^2}^2 + 2 \chi s + \int_0^s \langle g_r', b\rangle_{L^2} \,\dee \beta_r \, .
\]
By taking $\E$ of both sides, the stochastic integral vanishes, hence $\E \| g_s'\|_{L^2}^2 = \| g \|_{L^2}^2 + 2 \chi s$. Putting this together, we conclude for suitable $V(u)$
\[
	|P_1\phi(u,g) - P_1\phi(u^\prime,g)| \leq \|\phi\|_{\mathrm{Lip}}\E\|\tilde{g}_1\|_{H^{-1}} \leqc_p \|\phi\|_{\mathrm{Lip}}(V(u) + V(u^\prime))(1+\|g\|_{L^2})\|u-u^\prime\|_{\Hbf}.
\]
\end{proof}

\subsection{Stationarity of $\mu^0$: completing the proof}\label{subsec:concludeStationarity}

It suffices to show that the right-hand side of \eqref{eq:two-terms} vanishes. 
The first term $|\int (P^{\kappa_n}_1 \phi - P^0_1 \phi) \dee \mu^{\kappa_n}|$ in \eqref{eq:two-terms} will be estimated using the following. 

\begin{lemma}\label{lem:firstTermZeroSparrow}
Let $g_t^\kappa$ and $g_t^0$ have the same initial data $g \in L^2$, the same velocity field 
$(u_t)$, and the same source path $b\beta_t$. Then,
\[
	\E\|g_1^\kappa - g_1^0\|_{H^{-1}} \leqc V(u) \int_0^1 \kappa \E \|g^\kappa_s\|_{H^1}\ds.
\]
\end{lemma}
\begin{proof}
Note that $\tilde{g}_t^\kappa = g_t^\kappa - g_t^0$ satisfies
\[
	\partial_t \tilde{g}_t^\kappa + u_t\cdot \nabla \tilde{g}_t^\kappa - \kappa \Delta g_t^\kappa =0.
\]
with $\tilde{g}_0^\kappa =0$. This implies that
\[
	\tilde{g_1}^\kappa = \kappa \int_0^1 S_{t-s}^{0}(\theta_s\omega,u_s) \Delta g^\kappa_s\,\ds,
\]
and therefore using the $H^{-1}$ bound \eqref{eq:H-1bound} on $S_{t-s}^{0}(\theta_s\omega,u_s)$ we find
\[
	\|\tilde{g}_1^\kappa\|_{H^{-1}}\leqc \kappa \int_0^1 \bar D(\theta_s\omega,u_s) \|g^\kappa_s\|_{H^{1}}\ds \, , 
\]
where $\bar D(\omega, u)$ has $V$-bounded 1st moment. Taking an expectation w.r.t. $\omega$
and applying Lemma \ref{lem:TwistBd} completes the proof. 
\end{proof}

From here, to prove convergence $\int (P_1^{\kappa_n} \phi - P_1^0 \phi) d \mu^{\kappa_n} \to 0$,
we have from Lemma \ref{lem:firstTermZeroSparrow} and the Lipschitz property of $\phi$ that
\[
| (P_1^\kappa \phi - P_1^0 \phi )(u,g) | \leq \| \phi\|_{\rm Lip} \E \| g_1^\kappa - g_1^0\|_{H^{-1}}
\lesssim V(u) \int_0^1 \kappa \E \|g_s^\kappa\|_{H^{1}}\ds \, , 
\]
hence upon integrating in $\mu^\kappa$ and using Cauchy-Schwarz,
\begin{align}
	\int|P_1^\kappa\phi - P_1^0\phi|\,\dee\mu^\kappa \leqc \kappa \int_0^1 \left(\E_{\mu^\kappa}\|g\|_{H^1}^2\right)^{1/2}\ds \leqc \sqrt{\kappa}
\end{align}
 by the energy balance relation $\kappa \E_{\mu^\kappa} \|g_t\|_{H^1}^2 = \chi$ (see \eqref{eq:L2Balance}).

\medskip 

It remains to show that $| \int P_1^0 \phi (\dee \mu^{\kappa_n} - \dee \mu^0)| \to 0$ in \eqref{eq:two-terms}. 
As observed already, this does not follow immediately 
from weak convergence since $P_1^0 \phi$ need not be continuous. Instead, we will use
 the restricted Feller property from Lemma \ref{lem:restricted-feller}, which guarantees continuity
 of $P_1^0 \phi$ along $g \in L^2$. To bridge the gap from $L^2$ to $H^{-1}$, we use the following
 mollification argument: for $\delta \in (0,1)$ and $g \in H^{-s}, s \in (0,1]$, define $T_\delta g$ to be the mollification \footnote{Here mollification is defined by convolution with $\eta_\delta = \delta^{-d}\eta(\cdot/\delta)$, where $\eta$ is a smooth, symmetric, compactly supported test function with $\int \eta = 1$.} of $g$, and define the regularized semigroup
\[
	(P_1^0\phi)_{\delta}(u,g):= P_1^0\phi(u,T_\delta g) \, .
\]
Note that a straight-forward duality argument shows that
\[
	\|T_\delta g - g\|_{H^{-1}}\leqc \delta^{1-s}\|g\|_{H^{-s}}.
\]
This, in turn gives rise to the following approximation property for $(P^0_1\phi)_\delta$.
\begin{lemma}\label{lem:4712831}
For each $\delta \in (0,1), s \in (0,1]$, $(P_1^0\phi)_{\delta} \in C_b(\Hbf\times H^{-1})$ and for each globally Lipschitz $\phi : \Hbf \times H^{-1} \to \R$ and $g\in H^{-s}$, we have
\[
	|(P_1\phi)_\delta - P_1\phi|(u,g) \leqc V(u) \delta^{1-s}\|g\|_{H^{-s}},
\]
where $V = V_{\beta,\eta}$ for any $\beta > 0$ and $\eta>0$.
\end{lemma}
\begin{proof}
Note that Lemma \ref{lem:restricted-feller} implies $(P_1^0\phi)_{\delta} \in C_b(\Hbf\times H^{-1})$. For $g\in H^{-s}$, denote $g_1^\delta$ and $g_1$ solutions at time $t=1$ of the transport equation \eqref{eq:kappa=0-eq} with the same velocity path $(u_t)$ and noise $b\beta_t$ but with different initial data $T_\delta g$ and $g$ respectively. We have
\[
\begin{aligned}
	|((P_1^0\phi)_\delta - P_1^0\phi)(u,g)| &\leq \E |\phi(u_1,g_1^\delta) - \phi(u_1,g_1)|\\
	&\leq \|\phi\|_{\mathrm{Lip}}\E\|S_1^{0}(u,\omega)(T_\delta g - g)\|_{H^{-1}}\\
	&\leqc V(u) \|T_\delta g - g\|_{H^{-1}}\\
	&\leqc V(u) \delta^{1-s}\|g\|_{H^{-s}}
\end{aligned}
\]
having used \eqref{eq:H-1bound} and the $V$-boundedness of the 1st moment of $\bar D$.
\end{proof}
To complete the proof: for $\delta >0$, we estimate
\[
	\left|\int P_1^0 \phi\, (\dee \mu^{\kappa_n} - \dee \mu^0)\right| \leq \left|\int (P_1^0 \phi)_\delta\, (\dee \mu^{\kappa_n} - \dee \mu^0)\right| + \int |(P_1^0 \phi)_\delta - P_1^0\phi|\, (\dee \mu^{\kappa_n} + \dee \mu^0)
\]
The first term on the right vanishes as $n \to \infty$ since $(P_1\phi)_\delta\in C_b(\Hbf\times H^{-1})$ by Lemma \ref{subsec:restrictedFeller} and $\mu^{\kappa_n} \to \mu^0$ weakly. For the
second term, Lemma \ref{lem:4712831} and Cauchy-Schwarz imply 
\[
	\int |(P_1 \phi)_\delta - P_1\phi|\, (\dee \mu^{\kappa_n} + \dee \mu^0) \leqc \delta^{1-s}\left(\left(\E_{\mu^{\kappa_n}}\|g\|_{H^{-s}}^2\right)^{1/2} + \left(\E_{\mu^0}\|g\|_{H^{-s}}^2\right)^{1/2}\right)
\]
which converges to $0$ as $\delta \to 0$ uniformly in $\kappa_n$ by \eqref{eq:unifKappaEst}.
This completes the proof of stationarity of $\mu^0$ for $P_t^0$, the last item
remaining from Theorem \ref{thm:kappato0}.

\subsection{Non-vanishing flux}\label{subsec:Non-vanishing-flux}

We conclude this section by proving the non-vanishing flux law \eqref{eq:AD212}. 
To do this, we remark that $\Pi_{\leq N}g_t^\kappa$ satisfies
\[
	\Pi_{\leq N}g_t^\kappa = \Pi_{\leq N}g_0^\kappa - \int_0^t\Pi_{\leq N} (u_s\cdot \nabla g_s^\kappa)\ds + \Pi_{\leq N} b\beta_t, 
\]
and therefore applying It\^{o}'s formula gives
\begin{equation}
\begin{aligned}
	\|\Pi_{\leq N}g_t^0\|_{L^2}^2 &= \|\Pi_{\leq N}g_0^0\|^2_{L^2} - 2\int_0^t\langle \Pi_{\leq N}g_s, \Pi_{\leq N}(u_s\cdot\nabla g_s)\rangle_{L^2}\ds\\
	&\hspace{1in} + \|\Pi_{\leq N}b\|^2_{L^2}t + \int_0^t 2\langle \Pi_{\leq N}g_s,\Pi_{\leq N}b\rangle_{L^2}\dee \beta_s.
\end{aligned}
\end{equation}
Taking expectation, integrating the initial data $(u,g)$ with respect to the stationary measure $\mu^0$ and using the stationarity of $(u_t,g^0_t)$ with respect to $\mu^0$ we readily obtain the flux balance
\[
	\E_{\mu^{\kappa}}\langle \Pi_{\leq N}g, \Pi_{\leq N}(u\cdot\nabla g)\rangle_{L^2} = \frac{1}{2}\|\Pi_{\leq N}b\|_{L^2}^2.
\]
Using the divergence free property and integrating by parts in the $L^2$ inner product gives \eqref{eq:AD212}.

\section{Irregularity of the limiting statistically stationary solutions}\label{sec:onsager}
%!TEX root = master.tex
In the previous section, we produced a $\kappa \to 0$ subsequential limit of $\mu^{\kappa}$ which converges to a stationary measure $\mu^0$ of the Markov semigroup associated with  $(u_t,g_t^0)$, where the scalar solves 
\begin{equation}\label{eq:trans2}
	\partial_tg^0_t + u_t\cdot\nabla g^0_t = b\dot{\beta}_t \, . 
\end{equation}
Since there is no dissipation in the equation, the input from the noise must be ``anomalously'' dissipated by the mixing mechanism.
As discussed in Section \ref{subsubsec:oCritSubsc}, this requires a degree of roughness of statistically stationary $g_t^0$. Specifically, in this section we prove
\begin{itemize}
\item[(a)] $\mu^0$-generic
$(u,g) \in \Hbf \times H^- = \Hbf \times \cap_{\delta > 0} H^{-\delta}$ are such that $g$ is not locally integrable, hence
are `strictly' distributions (Theorem \ref{thm:Ocrit}); 
\item[(b)] the Besov space $L^2_tB^0_{2,\infty}$ is ``Onsager-type critical'', in the sense that no
statistically stationary solution $g_t^0$ can have even a small amount of additional regularity (Theorem \ref{thm:OcritB2inf}). 
\end{itemize}
Part (a) is carried out by applying a variant of the DiPerna-Lions theory of renormalized solutions for the transport equation, and is carried out in Section \ref{subsec:L1}, while part (b) is carried out in Section \ref{subsec:OcritFuckery}. 

\subsection{$\mu^0$ generic functions cannot be $L^1$}\label{subsec:L1}

Our proof is by contradiction, 
based on invariance of $\Hbf \times L^1$ by solutions
to \eqref{eq:trans2}. 
\begin{lemma}\label{lem:L1-inv}
The set $\Hbf\times L^1$ is almost surely invariant for $(u_t,g_t^0)$: if 
$u_0 \in \Hbf$ and $g_0^0 \in L^1$, then $g_t^0 \in L^1$ for all $t > 0$. 
\end{lemma}
\begin{proof}
This follows from Lemma \ref{lem:H-1-semigroup}(b) by noting that due to smoothness of $b$, 
the stochastic convolution
\[
	\int_0^tS_{t-s}^0(\theta_s\omega,u_s)b \,\dee \beta_s
\]
takes values in (at least) $L^2$, and $S_t^0$ propagates $L^1$ regularity. In fact, 
we have 
\[
\|S_t^0(\omega,u)g\|_{L^1} = \| g \circ (\phi^t_{\omega, u})^{-1}\|_{L^1} = \|g\|_{L^1}
\]
by the fact that $\phi^t_{\omega, u}$ is volume preserving.
\end{proof}

To continue, we will find it convenient to show that any $L^1$ valued solutions to \eqref{eq:trans2} can be renormalized in the sense of DiPerna and Lions \cite{DPL89}, meaning that we show that $F(g_t^0)$ solves another transport-type equation for some suitably regular function $F:\R \to \R$. In our setting the velocity field $(u_t)$ is not rough, which allows for the transport part of the equation to be easily renormalized, but the presence of noise introduces It\^{o} corrections and imposes higher regularity requirements on $F$. In what follows we define
\[
	F(z) := \sqrt{1+|z|^2}.
\]

\begin{lemma}
Let $g^0\in L^1$, and $g_t^0$ be a mild solution to \eqref{eq:trans2}, then the following holds almost surely
\begin{equation}\label{eq:renormalized-eq}
\int F(g_{t}^0)\dx = \int F(g_0^0)\dx + \int_0^{t}\left(\int\frac{b^2}{2F(g^0_s)^{3}}\dx \right)\ds + \int_0^{t}\left(\int\frac{g_s^0 b}{F(g^0_s)}\dx\right)\dee \beta_s.
\end{equation}
\end{lemma}
\begin{proof}
For a given function $f\in L^1$, we denote $(f)_{\delta}$ the mollification of $f$ (see the proof of Lemma \ref{lem:4712831}). It is straightforward to show that $(f)_{\delta} \to f$ both in $L^1$ and pointwise a.e..
 Since $g_t^0$ is a weak solution to \eqref{eq:trans2} we see that $(g_t^0)_{\delta}$ solves the following equation almost surely: 
\[
	(g_t^0)_{\delta}  = (g_0^0)_\delta  - \int_0^t u_s\cdot\nabla (g_s^0)_{\delta}\ds  + (b)_{\delta}\beta_t + \int_0^t R_\delta(u_s,g_s^0)\ds  \, , 
\]
where, we write 
\[
	R_\delta(u_t,g_t^0) = u_t\cdot\nabla (g_t^0)_\delta- (u_t\cdot\nabla g_t^0)_\delta
\]
for the commutator term, noting that convolution $(\cdot)_{\delta}$ does not commute with $u_t\cdot\nabla$. Since the equation for $(g_t^0)_\delta$ holds pointwise as an identity on smooth functions for each $x\in\T^d$, we can apply It\^{o}'s formula to $F((g_t^0)_\delta(x))$ to deduce the pointwise identity
\begin{equation}\label{eq:beta-reg-id}
\begin{aligned}
	F((g_t^0)_\delta) &= F((g_0^0)_\delta) - \int_0^t u_s \cdot \nabla F((g_s^0)_\delta)\,\ds + \int_0^t\frac{1}{2}F^{\prime\prime}((g_s^0)_\delta)(b)_\delta^2\, \ds + \int_0^t F^\prime((g_s^0)_\delta)(b)_\delta\,\dee \beta_t\\
	&\hspace{1in}+ \int_0^t F^\prime\left((g_s^0)_\delta\right)R_{\delta}(u_s,g_s^0)\ds.
\end{aligned}
\end{equation}
Above, we have used the fact $(g_t^0)_\delta$ is smooth and therefore $F^\prime(g_t^0) u_t\cdot\nabla(g_t^0)_\delta = u_t\cdot\nabla F((g_t^0)_\delta)$.
It\^o's formula for the unbounded function $F(z) = \sqrt{1+|z|^2}$ is justified in this case
because $F\left((g_t^0)_\delta\right) \lesssim_\delta 1 + \|g_t^0\|_{L^1}$, hence
for each $t\in \R_+$, 
\[
\P\left(\sup_{s\in[0,t]}\|F((g_s^0)_\delta)\|_{L^\infty} <\infty\right)	= 1.
\]
Next, we note that since $\Div u_t = 0$, it is standard that (see \cite{DPL89})
\[
	\int|R_{\delta}(u_t,g_t^0)|\dx \leq \|u_t\|_{W^{1,\infty}}\left(\int_{B_\delta} |y\cdot\nabla \eta_\delta| \dy\right) \sup_{|y|<\delta}\|g^0_t(\cdot+y) - g_t^0(\cdot)\|_{L^1} \, . 
  \]  
Since $\int |y\cdot\nabla \eta_\delta| \dy = \int |y\cdot\nabla \eta| \dy$, and $g_t^0$ belongs to $L^1$ almost surely we conclude by the $L^1$-continuity of spatial translations that for each $t\in \R_+$
\begin{equation}\label{eq:commutator}
	R_\delta(u_t,g_t^0) \to 0\quad \text{in $L^1$ as }\delta\to 0\text{ almost surely}.
\end{equation}
Using the fact
\[
	F^\prime(z) = \frac{z}{F(z)} ,\quad F^{\prime\prime}(z) = \frac{1}{F(z)^3}
\]
we can integrate \eqref{eq:beta-reg-id} and use the fact that $u_t$ is divergence free to obtain
	\begin{equation}\label{eq:beta-reg-id-1}
\begin{aligned}
	\int F((g_t^0)_\delta)\dx &= \int F((g_0^0)_\delta)\dx + \int_0^t\left(\int\frac{1}{2}\frac{(b)_\delta^2}{F((g_s^0)_\delta)^3 }\dx\right)\ds + \int_0^t\left(\int\frac{(b)_\delta(g_t^0)_{\delta}}{F((g_s^0)_\delta)}\dx\right)\dee \beta_t\\
	&\hspace{1in}+\int_0^t \left(\int\frac{(g_t^0)_\delta}{F\left((g_t^0)_\delta\right)}R_{\delta}(u_s,g_s^0)\dx\right)\ds.
\end{aligned}
\end{equation}
The proof will be complete if we can pass the $\delta\to 0$ limit in both sides of \eqref{eq:beta-reg-id-1} almost-surely for each fixed $t$. Since we have $(g_t^0)_{\delta} \to g_t^0$ in $L^1$ and almost everywhere on $\T^d$, as well as the convergence of the commutator \eqref{eq:commutator} in $L^1$ and $(b)_\delta \to b$ in $L^2$, we can use bounded convergence to  pass the $\delta \to 0$ limits almost-surely in every term of \eqref{eq:beta-reg-id-1} except for the stochastic integral. To deal with the stochastic integral, we note that by It\^{o}'s formula 
\[
	\E\left|\int_0^t\left(\int\frac{(b)_\delta(g_t^0)_{\delta}}{F((g_s^0)_\delta)} - \frac{b g_t^0}{F(g_s^0)}\dx\right)\dee \beta_t\right|^2 = \E\int_0^t\int\left|\frac{(b)_\delta(g_t^0)_{\delta}}{F((g_s^0)_\delta)} - \frac{b g_t^0}{F(g_s^0)}\right|\dx\ds \to 0
\]
as $\delta \to 0$ by the bounded convergence theorem. 
\end{proof}

We are now ready to prove Theorem \ref{thm:Ocrit}.

\begin{proof}[\textbf{Proof of Theorem \ref{thm:Ocrit}}]
In pursuit of a contradiction, we assume that $\mu^0(\Hbf\times L^1) > 0$. By Lemma \ref{lem:L1-inv}, $\Hbf \times L^1$ is an invariant set for $(u_t,g_t^0)$, and so the conditional measure
\[
	\hat{\mu}^0 :=  \frac{\mu^0(\cdot\cap (\Hbf\times L^1))}{\mu^0(\Hbf\times L^1)}
\]
is another stationary measure for $(u_t,g_t^0)$ that assigns full measure to $\Hbf\times L^1$. 

We would like to conclude the proof by taking expectations of both sides of \eqref{eq:renormalized-eq} using stationarity with respect to $\hat{\mu}^0$. However, we are unable to justify this, since
a priori we do not know whether or not $F(g_t^0)$ has finite moments. 
To get around this, we apply It\^{o}'s formula to $\phi\left(\int F(g_t^0)\dx\right)$, where $\phi$ is a bounded $C^2$ function, to deduce
\begin{equation}
\begin{aligned}
	\phi\left(\int F(g_t^0)\dx\right) &= \phi\left(\int F(g_0^0)\dx\right)+ \int_0^{t}\left(\int \frac{b^2}{2F(g^0_s)^{3}}\dx \right)\phi^\prime\left(\int F(g_s^0)\dx\right)\ds\\
	&+ \frac{1}{2}\int_0^t\left(\int \frac{g_s^0 b}{\beta(g^0_s)}\dx\right)^2\phi^{\prime\prime}\left(\int F(g_s^0)\dx\right)\ds\\
	&+ \int_0^{t} \left(\int \frac{g_s^0 b}{F(g^0_s)}\dx\right)\phi^{\prime}\left(\int F(g_s^0)\dx\right) \dee\beta_s. 
\end{aligned}
\end{equation}
Taking expectation and integrating the initial data with respect to $\hat{\mu}^0$ (crucially using that $\hat{\mu}^0$ assigns full measure to $L^1$) , we find
\[
	\E_{\hat{\mu}^0} \left[\left(\int \frac{b^2}{F(g)^{3}}\dx \right)\phi^\prime\left(\int F(g)\dx\right) + \left(\int \frac{g b}{F(g)}\dx\right)^2\phi^{\prime\prime}\left(\int F(g)\dx\right)\right] = 0. 
\]
For $\ep >0$, define $\phi(z) = \frac{z}{1+\ep z}$, so that
\[
	\E_{\hat{\mu}^0}\frac{\int \frac{b^2}{F(g)^{3}}\dx}{\left(1 + \ep \int F(g)\dx\right)^{2}} =  \E_{\hat{\mu}^0}\frac{2\ep\left(\int \frac{g b}{F(g)}\dx\right)^2}{\left(1+ \ep\int F(g)\dx\right)^{3}} \leq 2\ep\|b\|_{L^2}^2\, .
\]
Sending $\ep \to 0$ and applying monotone convergence gives
\[
	\E_{\hat{\mu}^0}\int \frac{b^2}{F(g)^{3}}\dx = 0 \, ,
\]
which contradicts $\hat{\mu}^0(\Hbf \times L^1) = 1$. Therefore $\mu^0(\Hbf \times L^1) = 0$, which
completes the proof. 
\end{proof}

\subsection{Onsager-type criticality of $L^2_t B_{2,\infty}^0$}\label{subsec:OcritFuckery}
In this section we prove Theorem \ref{thm:OcritB2inf}. 
Recall Definition \ref{def:gen-mult} of a $B^{0}_{2,c}$-suitable multiplier $M$, and the associated generalized Besov norm
\[
	\|f\|_{B^{M}_{2,\infty}} = \sup_{N \in \{2^j: j \in \N_*\}} M(N)\|\Pi_Nf\|_{L^2}.
\]
We denote $B^M_{2,\infty}$ the associated space of tempered distributions $f$ with $\|f\|_{B^{M}_{2,\infty}} < \infty$. The following Lemma shows that we can exhaust $B^{0}_{2,c}$ in terms of the spaces $\{B^{M}_{2,\infty}\}$.
\begin{lemma}\label{lem:quantRegBesov}
The following holds
\[
	B^0_{2,c} = \bigcup_{M} B^{M}_{2,\infty},
\]
where the union is over all $B^{0}_{2,c}$-suitable multipliers $M$.
\end{lemma}
\begin{proof}
That $\norm{f}_{B_{2,\infty}^M} < \infty$ $\Rightarrow$ $f \in B^0_{2,c}$ follows from $\lim_{k \to \infty} M(k) = \infty$. 

Conversely, let $g \in B_{2,c}^0$. Let $C_0$ be such that $\sup_{j \in \N_\ast} \norm{\Pi_{2^j} g}_{L^2} \leq C_0$. 
By $g \in B_{2,c}^0$, there exists a strictly increasing sequence $\set{N_k}_{k=1}^\infty \subseteq \set{2^j : j \in \mathbb N}$ such that  $\forall N \geq N_k$,  $\norm{\Pi_N g}_{L^2} < 2^{-k}$. 
Define $M(k): [1,\infty) \to [1,\infty)$ to be the monotone increasing multiplier
\begin{align}
M(k ) = 
\begin{cases}
2  \quad k < N_1 \\
\frac{2^j}{N_{j+1} - N_j} \left(N_{j+1} - k \right) + \frac{2^{j+1}}{N_{j+1} - N_{j}}\left(k - N_{j}\right) \quad  N_j \leq k \leq N_{j+1}. 
\end{cases}
\end{align}
We see that $M$ is piecewise linear with slope $\leq 1$, and therefore satisfies
conditions (i) and (ii) of Definition \ref{def:gen-mult}. 
\end{proof}

We introduce Littlewood-Paley decomposition for future use.
Let $\zeta \in C_0^\infty(\R;\R)$ be such that $\zeta(\xi) = 1$ for $\abs{\xi} \leq 1$ and $\zeta(\xi) = 0$ for $\abs{\xi} \geq 3/2$ and define $\psi(\xi) = \zeta(\xi/2) - \zeta(\xi)$, supported in the range $\xi \in (1,3)$. Denote for each $N > 0$, $\zeta_N(\xi)  := \zeta(N^{-1}\xi)$ and $\psi_N(\xi) := \psi(N^{-1}\xi)$. 
For $f \in L^2(\R)$ we define the Littlewood-Paley projections
\begin{align}
\pi_Nf    := \psi_N(\grad)f \quad\text{and}\quad \pi_{\leq N} f  := \zeta_{N}(\grad) f. 
\end{align}
Note by Definition \ref{def:gen-mult}, (denoting $\N_\ast = \set{0} \cup \N$ and analogously for $\Z_\ast^d = \set{0} \cup \Z^d$ below). 
\begin{align}
\sup_{N \in \set{2^j: j \in \N_\ast}} M(N)\norm{\Pi_N g}_{L^2} \approx \norm{M(|\grad|) \pi_{\leq 1} g}_{L^2} + \sup_{N \in \set{2^j: j\in\N_\ast}}  \norm{M(|\grad|)\pi_N g}_{L^2}.
\end{align}

\begin{lemma} \label{lem:FluxFan}
For any $(u,g) \in \Hbf \times H^{-1}$, if $\norm{g}_{B_{2,\infty}^M} < \infty$ for some $B_{2,c}^0$-suitable $M$, then 
\begin{align}
\abs{\brak{\pi_{\leq N}g, \pi_{\leq N}\grad \cdot (u g)}} \lesssim \frac{1}{M^2(N)}\norm{g}_{B_{2,\infty}^M}^2\norm{u}_{\Hbf}. 
\end{align}
\end{lemma}
\begin{proof}
By incompressibility of $u$ and Plancherel's identity
\begin{align*}
  \brak{\pi_{\leq N}g, \pi_{\leq N}\grad \cdot (u g)} & \\& \hspace{-2cm}
  = \brak{\pi_{\leq N}g, \pi_{\leq N}\grad \cdot (u g) - \grad \cdot (u \pi_{\leq N} g)} & \\& \hspace{-2cm}
  = \sum_{k,\ell \in \Z^d_\ast} \left(\mathbf{1}_{\abs{\ell} < 2\abs{k-\ell}} + \mathbf{1}_{\abs{\ell} \geq 2\abs{k-\ell}}\right)\zeta_{N}(k)\overline{\hat{g}_t}(k) \left(\zeta_{N}(k) - \zeta_{N}(\ell) \right) ik \cdot \hat{u}_t(k-\ell) \hat{g}_{t}(\ell) \\
&\hspace{-2cm} =: \mathcal{I}_{HL} + \mathcal{I}_{LH}. 
\end{align*}
Note that due to the presence of the cutoffs $\left(\zeta_{N}(k) - \zeta_{N}(\ell) \right)$ we see that one of $\abs{k}$ or $\abs{\ell}$ must be larger or equal to $\frac{1}{2}N$ for the corresponding term in the summation to be non-zero.
The ``high-low'' term is treated by noting that on the support of the summation, $\abs{k} + \abs{\ell} \lesssim \abs{k-\ell}$, and therefore for any $\delta > 0$,
\begin{align}
\abs{\mathcal{I}_{HL}} & \lesssim \frac{1}{N^2}\sum_{k,\ell \in \Z^d_\ast} \mathbf{1}_{\abs{\ell} < 2\abs{k-\ell}} \abs{\zeta_{N}(k)\hat{g}_t(k) \brak{k-\ell}^3\hat{u}_t(k-\ell) \hat{g}_{t}(\ell)} \\
& \lesssim \frac{1}{N^2}\sum_{k,\ell \in \Z^d_\ast} \mathbf{1}_{\abs{\ell} < 2\abs{k-\ell}} \frac{1}{\brak{k}^{\delta} \brak{\ell}^{\delta}}\abs{\zeta_{N}(k) \hat{g}_t(k) \brak{k-\ell}^{3+2\delta} \hat{u}_t(k-\ell) \hat{g}_{t}(\ell)}. 
\end{align}
Then, using Cauchy-Schwarz, Young's inequality, and  $H^{r} \hookrightarrow \widehat{L^1}$ for $r > d/2$, (recall $\sigma > 3+ \frac{d}{2}$), 
\begin{align}
  \abs{\mathcal{I}_{HL}} & \lesssim_\delta \frac{1}{N^2} \norm{g}_{H^{-\delta}}^2 \norm{u}_{\Hbf} \lesssim \frac{1}{M(N)^2} \norm{g}_{B_{2,\infty}^M}^2 \norm{u}_{\Hbf},
\end{align}
where the last inequality followed from $M(N) \lesssim N$ by Definition \ref{def:gen-mult} and the embedding $H^{-\delta} \hookrightarrow B_{2,\infty}^0$. 

We turn next to the ``low-high'' term. 
First, by the mean value theorem there holds, 
\begin{align}
\abs{\zeta_{N}(k) - \zeta_{N}(\ell)} \lesssim \frac{1}{N}\abs{k-\ell}. \label{ineq:MVTpsi}
\end{align} 
Second, observe that on the support of the summation, we have $N \approx \abs{k} \approx \abs{\ell}$ because of the frequency cut-off and
that at least one of $\abs{k}$ or $\abs{\ell}$ must be larger or equal to $\frac{1}{2}N$ but one of $\abs{k}$ and $\abs{\ell}$ must also be less than $3N$. 
Therefore, we deduce from part (ii) of Definition \ref{def:gen-mult} and \eqref{ineq:MVTpsi},
\begin{align}
\abs{\mathcal{I}_{LH}}&  \lesssim \frac{1}{M(N)^2}\sum_{k,\ell \in \Z^d_\ast} \mathbf{1}_{\abs{\ell} \geq 2\abs{k-\ell}}\abs{\zeta_{N}(k)M(|k|)\hat{g}(k) \brak{k-\ell}\hat{u}(k-\ell) M(|\ell|)\hat{g}(\ell)}. 
\end{align}
Again by using Cauchy-Schwarz, Young's inequality, and $H^{r} \hookrightarrow \widehat{L^1}$ for $r > d/2$, 
\begin{align}
\abs{\mathcal{I}_{LH}} & \lesssim \frac{1}{M(N)^2} \norm{g}_{B_{2,\infty}^M}^2 \norm{u}_{\Hbf}. 
\end{align}
This completes the proof. 
\end{proof} 

\begin{proof}[\textbf{Proof of Theorem \ref{thm:OcritB2inf}}]
We proceed by contradiction and assume $\EE_{\mu^0} \norm{g}_{B_{2,\infty}^M}^p <\infty$ for some $p > 2$. 
By repeating the proof of the flux balance \eqref{eq:AD212} as in Section \ref{subsec:Non-vanishing-flux} with $\pi_{\leq N}$ replacing $\Pi_{\leq N}$ we have
\begin{align}
\E_{\mu^0}\brak{\pi_{\leq N}g, \pi_{\leq N}\grad \cdot (u g)}  =  \frac{1}{2}\norm{\pi_{\leq N} b}_{L^2}^2. \label{eq:piNbal}
\end{align}
However Lemma \ref{lem:FluxFan} and the assumption that $\EE_{\mu^0} \norm{g}_{B_{2,\infty}^M}^p <\infty$ for some $p > 2$ implies that
\begin{align}
\abs{\E_{\mu^0}\brak{\pi_{\leq N}g, \pi_{\leq N}\grad \cdot (u g)}} & \lesssim \frac{1}{M(N)^2}\E_{\mu^0} \left[\norm{g}_{B_{2,\infty}^M}^2\norm{u}_{\Hbf}\right]\\
& \lesssim \frac{1}{M(N)^2}\left(\EE_{\mu^0} \norm{g}_{B_{2,\infty}^M}^{p}\right)^{2/p} \left(\EE \norm{u}_{\Hbf}^{\frac{p}{p-2}}\right)^{\frac{p-2}{p}} \\
& \lesssim \frac{1}{M(N)^2}.
\end{align}
On the other hand, $\lim_{N \to \infty} \frac{1}{2}\norm{\pi_{\leq N} b}_{L^2}^2 = \chi$ and so by choosing $N$ sufficiently large in \eqref{eq:piNbal} gives the desired contradiction.  
\end{proof}

%% Bibliography
\bibliographystyle{abbrv}
\bibliography{bibliography}

\end{document}